\documentclass{article}%
\usepackage{amsmath}
\usepackage{amsfonts}
\usepackage{amssymb}
\usepackage{graphicx}%
\newtheorem{theorem}{Theorem}

\newtheorem{corollary}[theorem]{Corollary}

\newtheorem{definition}[theorem]{Definition}

\newtheorem{lemma}[theorem]{Lemma}

\newtheorem{proposition}[theorem]{Proposition}
\newtheorem{remark}[theorem]{Remark}

\newenvironment{proof}[1][Proof]{\noindent\textbf{#1.} }{\ \rule{0.5em}{0.5em}}
\begin{document}

\title{Global $L^{p}$ estimates for degenerate Ornstein-Uhlenbeck
operators\thanks{2000 MSC: primary 35H10; secondary 35B45, 35K70, 42B20. Key
words: Ornstein-Uhlenbeck operators; global $L^{p}$-estimates; hypoelliptic
operators; singular integrals; nonhomogeneous spaces. }}
\author{Marco Bramanti\\Dip. di Matematica, Politecnico di Milano\\Via Bonardi 9, 20133 Milano, Italy\\marco.bramanti@polimi.it
\and Giovanni Cupini\\Dip. di Matematica \textquotedblleft Ulisse Dini\textquotedblright%
\ Universit\`{a} degli Studi di Firenze \\Viale Morgagni 67/A, 50134 Firenze, Italy \\cupini@math.unifi.it
\and Ermanno Lanconelli\\Dip. di Matematica. Universit\`{a} di Bologna.\\Piazza di Porta S. Donato 5, 40126 Bologna, Italy\\lanconel@dm.unibo.it
\and Enrico Priola\\Dip. di Matematica, Universit\`{a} di Torino\\via Carlo Alberto 10, 10123 Torino, Italy\\enrico.priola@unito.it}
\maketitle

\begin{abstract}
We consider a class of degenerate Ornstein-Uhlenbeck operators in
$\mathbb{R}^{N}$, of the kind%
\[
\mathcal{A}\equiv\sum_{i,j=1}^{p_{0}}a_{ij}\partial_{x_{i}x_{j}}^{2}%
+\sum_{i,j=1}^{N}b_{ij}x_{i}\partial_{x_{j}}%
\]
where $\left(  a_{ij}\right)  ,\left(  b_{ij}\right)  $ are constant matrices,
$\left(  a_{ij}\right)  $ is symmetric positive definite on $\mathbb{R}%
^{p_{0}}$ ($p_{0}\leq N$), and $\left(  b_{ij}\right)  $ is such that
$\mathcal{A}$ is hypoelliptic. For this class of operators we prove global
$L^{p}$ estimates ($1<p<\infty$) of the kind:%
\[
\left\Vert \partial_{x_{i}x_{j}}^{2}u\right\Vert _{L^{p}\left(  \mathbb{R}%
^{N}\right)  }\leq c\left\{  \left\Vert \mathcal{A}u\right\Vert _{L^{p}\left(
\mathbb{R}^{N}\right)  }+\left\Vert u\right\Vert _{L^{p}\left(  \mathbb{R}%
^{N}\right)  }\right\}  \text{ for }i,j=1,2,...,p_{0}%
\]
and corresponding weak (1,1) estimates. This result seems to be the first case
of global estimates, in Lebesgue $L^{p}$ spaces, for complete H\"{o}rmander's
operators%
\[
\sum X_{i}^{2}+X_{0},
\]
proved in absence of a structure of homogeneous group. We obtain the previous
estimates as a byproduct of the following one, which is of interest in its
own:%
\[
\left\Vert \partial_{x_{i}x_{j}}^{2}u\right\Vert _{L^{p}\left(  S\right)
}\leq c\left\Vert Lu\right\Vert _{L^{p}\left(  S\right)  }%
\]
for any $u\in C_{0}^{\infty}\left(  S\right)  ,$ where $S$ is the strip
$\mathbb{R}^{N}\times\left[  -1,1\right]  $ and $L$ is the
Kolmogorov-Fokker-Planck operator $\mathcal{A}-\partial_{t}.$ To get this
estimate we crucially use the left translation invariance of $L$ on a Lie
group $\mathcal{K}$ in $\mathbb{R}^{N+1}$ and some results on singular
integrals on nonhomogeneous spaces recently proved in \cite{B}.

\end{abstract}

\section{Introduction\label{section introduction}}

\subsection*{\textbf{Problem and main result}}

Let us consider the class of degenerate Ornstein-Uhlenbeck operators in
$\mathbb{R}^{N}$:%
\[
\mathcal{A}=div\left(  A\nabla\right)  +\left\langle x,B\nabla\right\rangle
=\sum_{i,j=1}^{N}a_{ij}\partial_{x_{i}x_{j}}^{2}+\sum_{i,j=1}^{N}b_{ij}%
x_{i}\partial_{x_{j}},
\]
where $A$ and $B$ are constant $N\times N$ matrices, $A$ is symmetric and
positive semidefinite. If we define the matrix:%
\begin{equation}
C\left(  t\right)  =\int_{0}^{t}E\left(  s\right)  AE^{T}\left(  s\right)
ds\text{, where }E\left(  s\right)  =\exp\left(  -sB^{T}\right)  \label{C(t)}%
\end{equation}
then it can be proved (see \cite{LP}) the equivalence between the three conditions:

- the operator $\mathcal{A}$ is hypoelliptic;

- $C\left(  t\right)  >0$ for any $t>0;$

- the following H\"{o}rmander's condition holds:%
\[
\text{rank} \mathcal{L}\left(  X_{1},X_{2},...,X_{N},Y_{0}\right)  =N,\text{
\ \ for any }x\in\mathbb{R}^{N},
\]
where
\begin{align*}
Y_{0}  &  =\left\langle x,B\nabla\right\rangle \text{ \ \ \ and}\\
\ X_{i}  &  =\sum_{j=1}^{N}a_{ij}\partial_{x_{j}}\text{ \ }i=1,2,...,N.
\end{align*}

Under one of these conditions it is proved in \cite{LP} that, for some basis
of $\mathbb{R}^{N}$, the matrices $A,B$ take the following form:%
\begin{equation}
A=%
\begin{bmatrix}
A_{0} & 0\\
0 & 0
\end{bmatrix}
, \label{A}%
\end{equation}
with $A_{0}=\left(  a_{ij}\right)  _{i,j=1}^{p_{0}}$ $p_{0}\times p_{0}$
constant matrix ($p_{0}\leq N$), symmetric and positive definite:%
\begin{equation}
\nu\left\vert \xi\right\vert ^{2}\leq\sum_{i,j=1}^{p_{0}}a_{ij}\xi_{i}\xi
_{j}\leq\frac{1}{\nu}\left\vert \xi\right\vert ^{2} \label{ellitpicity}%
\end{equation}
for any $\xi\in\mathbb{R}^{p_{0}},$ some positive constant $\nu$;%
\begin{equation}
B=%
\begin{bmatrix}
\ast & B_{1} & 0 & \ldots & 0\\
\ast & \ast & B_{2} & \ldots & 0\\
\vdots & \vdots & \vdots & \ddots & \vdots\\
\ast & \ast & \ast & \ldots & B_{r}\\
\ast & \ast & \ast & \ldots & \ast
\end{bmatrix}
\label{B}%
\end{equation}
where $B_{j}$ is a $p_{j-1}\times p_{j}$ block with rank $p_{j},j=1,2,...,r$,
$p_{0}\geq p_{1}\geq...\geq p_{r}\geq1$ and $p_{0}+p_{1}+...+p_{r}=N$.

In this paper we consider hypoelliptic degenerate Ornstein-Uhlenbeck
operators, with the matrices $A,B$ already written as (\ref{A}) and (\ref{B}).
For this class of operators, we will prove the following global $L^{p}$ estimates:

\begin{theorem}
\label{Thm main}For any $p\in\left(  1,\infty\right)  $ there exists a
constant $c>0,$ depending on $p,N,p_{0}$, the matrix $B$ and the number $\nu$
in (\ref{ellitpicity}) such that for any $u\in C_{0}^{\infty}\left(
\mathbb{R}^{N}\right)  $ one has:
\begin{align}
\left\Vert \partial_{x_{i}x_{j}}^{2}u\right\Vert _{L^{p}\left(  \mathbb{R}%
^{N}\right)  }  &  \leq c\left\{  \left\Vert \mathcal{A}u\right\Vert
_{L^{p}\left(  \mathbb{R}^{N}\right)  }+\left\Vert u\right\Vert _{L^{p}\left(
\mathbb{R}^{N}\right)  }\right\}  \text{ for }i,j=1,2,...,p_{0}%
\label{stima L_0 a}\\
\left\Vert Y_{0}u\right\Vert _{L^{p}\left(  \mathbb{R}^{N}\right)  }  &  \leq
c\left\{  \left\Vert \mathcal{A}u\right\Vert _{L^{p}\left(  \mathbb{R}%
^{N}\right)  }+\left\Vert u\right\Vert _{L^{p}\left(  \mathbb{R}^{N}\right)
}\right\}  . \label{stima L_0 b}%
\end{align}
Moreover, the following weak $\left(  1,1\right)  $ estimates hold:%
\begin{align}
\left\vert \left\{  x\in\mathbb{R}^{N}:\left\vert \partial_{x_{i}x_{j}}%
^{2}u\left(  x\right)  \right\vert >\alpha\right\}  \right\vert  &  \leq
\frac{c_{1}}{\alpha}\left\{  \left\Vert \mathcal{A}u\right\Vert _{L^{1}\left(
\mathbb{R}^{N}\right)  }+\left\Vert u\right\Vert _{L^{1}\left(  \mathbb{R}%
^{N}\right)  }\right\} \label{weak a}\\
\left\vert \left\{  x\in\mathbb{R}^{N}:\left\vert Y_{0}u\left(  x\right)
\right\vert >\alpha\right\}  \right\vert  &  \leq\frac{c_{1}}{\alpha}\left\{
\left\Vert \mathcal{A}u\right\Vert _{L^{1}\left(  \mathbb{R}^{N}\right)
}+\left\Vert u\right\Vert _{L^{1}\left(  \mathbb{R}^{N}\right)  }\right\}
\label{weak b}%
\end{align}
for any $\alpha>0,$ some constant $c_{1}$ depending on $N,p_{0},B$ and $\nu.$
\end{theorem}

Global estimates in H\"{o}lder spaces analogous to (\ref{stima L_0 a}%
)-(\ref{stima L_0 b}) have been proved by Da Prato and Lunardi \cite{DL} in
the nondegenerate case $p_{0}=N$ (corresponding to the classical
Ornstein-Uhlenbeck operator) and by Lunardi \cite{Lu} in the degenerate case;
$L^{p}$ estimates in the nondegenerate case $p_{0}=N$ have been proved by
Metafune, Pr\"{u}ss, Rhandi and Schnaubelt \cite{MPRS} by a semigroup
approach. Note that, even in the nondegenerate case, global estimates in
$L^{p}$ or H\"{o}lder spaces are not straightforward, due to the unboundedness
of the first order coefficients. Under this regard, our weak (1,1) estimate
seems to be new even in the nondegenerate case. $L^{2}$ estimates with respect
to an invariant Gaussian measure have been proved by Lunardi \cite{Lu2} in the
nondegenerate case, and by Farkas and Lunardi \cite{FL} in the degenerate case.

The operator $\mathcal{A}$ can be seen as the infinitesimal generator of the
Ornstein-Uhlenbeck semigroup. This is the Markov semigroup associated to the
stochastic differential equation:%
\begin{equation}
d\xi(t)=B^{T}\xi(t)dt+\sqrt{2}\,A_{0}^{1/2}dW\left(  t\right)
,\;\;\;t>0,\text{ \ \ }\xi(0)=x,\label{Langevin}%
\end{equation}
where $W\left(  t\right)  $ is a standard Brownian motion taking values in
$\mathbb{R}^{p_{0}}$. This equation can describe the random motion of a
particle in a fluid (see \cite{OU}). Several interpretations in physics and
finance for the operator $\mathcal{A}$ or its evolutionary counterpart $L$
(see below) are explained in the survey by Pascucci \cite{P}. Nonlocal
Ornstein-Uhlenbeck operators are studied by Priola and Zabczyk \cite{PZ}. In
infinite dimension, Ornstein-Uhlenbeck type operators arise naturally in the
study of stochastic P.D.E.s (see \cite{DZ} and \cite{DZ1}, \cite{CG} and the
references therein).

\begin{remark}
To make easier a comparison of our setting with that considered in several
papers we have quoted so far, we point out the fact that the condition
$C(t)>0$ is equivalent to the condition%
\[
Q_{t}\equiv\int_{0}^{t}\exp\left(  sB^{T}\right)  A\exp\left(  sB\right)
ds=\exp\left(  tB^{T}\right)  \,C(t)\,\exp\left(  tB\right)  >0.
\]
The operator $Q_{t}$ has also control theoretic meaning and is considered in
\cite{DL}, \cite{DZ}, \cite{DZ1}, \cite{FL}, \cite{MPRS}, \cite{PZ}. Also,
note that it is enough to require that $C(t)$ or $Q_{t}$ is positive definite
for some $t_{0}>0$ in order to get that it is positive definite for any $t>0$.
\end{remark}

\subsection*{\textbf{Relation with the evolution operator}}

The evolution operator corresponding to $\mathcal{A},$%
\[
L=\mathcal{A}-\partial_{t}\text{,}%
\]
is a Kolmogorov-Fokker-Planck ultraparabolic operator, which has been
extensively studied in the last fifteen years. The largest part of the related
literature is devoted to the case where an underlying structure of homogeneous
group is present. In absence of this structure (that is, in the general
situation we are interested in), this operator has been studied for instance
by Lanconelli and Polidoro \cite{LP}, Di Francesco and Polidoro \cite{DP},
Cinti, Pascucci and Polidoro \cite{CPP} (see also the survey \cite{LPP}, and
references therein). In particular, it is proved in \cite{LP} that the
operator $L$ is left invariant with respect to the Lie-group translation
\begin{align*}
\left(  x,t\right)  \circ\left(  \xi,\tau\right)   &  =\left(  \xi+E\left(
\tau\right)  x,t+\tau\right)  ;\\
\left(  \xi,\tau\right)  ^{-1}  &  =\left(  -E\left(  -\tau\right)  \xi
,-\tau\right)  \text{, where}\\
E\left(  \tau\right)   &  =\exp\left(  -\tau B^{T}\right)  .
\end{align*}

We will deduce global estimates (\ref{stima L_0 a}) from an analogous estimate
for $L$ on the strip
\[
S\equiv\mathbb{R}^{N}\times\left[  -1,1\right]  ,
\]
which can be of independent interest:

\begin{theorem}
\label{Thm main thm strip}For any $p\in\left(  1,\infty\right)  $ there exists
a constant $c>0$ such that%
\begin{equation}
\left\Vert \partial_{x_{i}x_{j}}^{2}u\right\Vert _{L^{p}\left(  S\right)
}\leq c\left\Vert Lu\right\Vert _{L^{p}\left(  S\right)  }\ \text{ for
}i,j=1,2,...,p_{0}, \label{Lu}%
\end{equation}
for any $u\in C_{0}^{\infty}\left(  S\right)  .$ The constant $c$ depends on
the same parameters than the $c$ in Theorem \ref{Thm main}.
\end{theorem}

To get the above $L^{p}$ estimates, we have to set the problem in the suitable
geometric framework, which for this specific class of operators has been
studied in detail in \cite{LP}, \cite{DP}, while for general H\"{o}rmander's
operators, with or without an underlying structure of homogeneous group, has
been investigated by Folland \cite{Fo}, Rothschild and Stein \cite{RS}, respectively.

In particular, $L^{p}$ estimates for the second order derivatives have been
proved in \cite{Fo} on the whole space, but assuming the existence of a
homogeneous group, and in \cite{RS} in the general case, but only locally.
Therefore our results cannot be deduced by the existing theories.

Actually, Theorem \ref{Thm main} seems to be the first case of global
estimates, in Lebesgue $L^{p}$ spaces, for hypoelliptic degenerate
Ornstein-Uhlenbeck operators, and more generally for complete H\"{o}rmander's
operators%
\[
\sum X_{i}^{2}+X_{0},
\]
in absence of an underlying structure of homogeneous group. We also want to
stress that the group $\mathcal{K}=\left(  \mathbb{R}^{N+1},\circ\right)  $ is
not in general nilpotent. Hence, in view of the results in \cite{tERS}, one
cannot expect a global $L^{p}$ estimate like (\ref{Lu}) to be true on the
whole $\mathbb{R}^{N+1}$ (instead that on a strip).

Our result can also be seen as a first step to study existence and uniqueness
for the Cauchy problem related to $L$ in $L^{p}$ spaces, as well as to
characterize the domain of the generator of the Ornstein-Uhlenbeck semigroup
\textbf{ }in $L^{p}$ spaces. We plan to address these problems in the next future.

\subsection*{\textbf{Strategy of the proof}}

Let us start noting that Theorem \ref{Thm main thm strip} easily implies
Theorem \ref{Thm main}, apart from the weak estimates (\ref{weak a}),
(\ref{weak b}), which will be proved separately. Namely, let%
\[
\psi\in C_{0}^{\infty}\left(  \mathbb{R}\right)
\]
be a cutoff function fixed once and for all, sprt$\,\psi\subset\left[
-1,1\right]  ,$ $\int_{-1}^{1}\psi\left(  t\right)  dt>0$. If $u:\mathbb{R}%
^{N}\rightarrow\mathbb{R}$ is a $C_{0}^{\infty}$ solution to the equation%
\[
\mathcal{A}u=f\text{ in }\mathbb{R}^{N},
\]
for some $f\in L^{p}\left(  \mathbb{R}^{N}\right)  $, let%
\[
U\left(  x,t\right)  =u\left(  x\right)  \psi\left(  t\right)  ;
\]
then%
\[
LU\left(  x,t\right)  =f\left(  x\right)  \psi\left(  t\right)  -u\left(
x\right)  \psi^{\prime}\left(  t\right)  \equiv F\left(  x,t\right)  .
\]
Therefore Theorem \ref{Thm main thm strip} applied to $U$ gives%
\begin{equation}
\left\Vert \partial_{x_{i}x_{j}}^{2}U\right\Vert _{L^{p}\left(  S\right)
}\leq c\left\Vert F\right\Vert _{L^{p}\left(  S\right)  }\ \ \text{ for
}i,j=1,2,...,p_{0}\label{stima strip}%
\end{equation}
hence%
\[
\left\Vert \partial_{x_{i}x_{j}}^{2}u\right\Vert _{L^{p}\left(  \mathbb{R}%
^{N}\right)  }\leq c\left\{  \left\Vert f\right\Vert _{L^{p}\left(
\mathbb{R}^{N}\right)  }+\left\Vert u\right\Vert _{L^{p}\left(  \mathbb{R}%
^{N}\right)  }\right\}
\]
with $c$ also depending on $\psi.$ Note that (\ref{stima L_0 b}) follows from
(\ref{stima L_0 a}).

\bigskip

We would like to describe now the general strategy of the proof of Theorem
\ref{Thm main thm strip}, as well as the main difficulties encountered. A
basic idea is that of linking the properties of $L$ to those of another
operator of the same kind, which not only is left invariant with respect to a
suitable Lie group of translations, but is also homogeneous of degree 2 with
respect to a family of dilations (which are group automorphisms). Such an
operator $L_{0}$ (see (\ref{L_0})) always exists under our assumptions, by
\cite{LP}, and has been called \textquotedblleft the principal
part\textquotedblright\ of $L.$ Note that the operator $L_{0}$ fits the
assumptions of Folland's theory \cite{Fo}. However, to get the desired
conclusion on $L,$ this is not enough. Instead, we exploit the fact that, by
results proved by \cite{DP}, the operator $L$ possesses a fundamental solution
$\Gamma$ with some good properties. First of all, $\Gamma$ is translation
invariant and has a fast decay at infinity, in space; this allows to reduce
the desired $L^{p}$ estimates to estimates on a singular integral operator
whose kernel vanishes far off the pole. Second, this singular kernel, which
has the form $\eta\cdot\partial_{x_{i}x_{j}}^{2}\Gamma$ where $\eta$ is a
radial cutoff function, satisfies \textquotedblleft standard
estimates\textquotedblright\ (in the language of singular integrals theory)
with respect to a suitable \textquotedblleft local quasisymmetric
quasidistance\textquotedblright\ $d$, which is a key geometrical object in our
study. Namely,
\[
d\left(  z,\zeta\right)  =\left\Vert \zeta^{-1}\circ z\right\Vert
\]
where $\zeta^{-1}\circ z$ is the Lie group operation related to the operator
$L,$ while $\left\Vert \cdot\right\Vert $ is a homogeneous norm related to the
principal part operator $L_{0}$ (recall that $L$ does not have an associated
family of dilations, and therefore does not have a natural homogeneous norm).
This \textquotedblleft hybrid\textquotedblright\ quasidistance is not (and
seemingly is not equivalent to) the control distance of any family of vector
fields; even worse, it does not fulfill enough good properties in order to
apply the standard theory of \textquotedblleft singular integrals in spaces of
homogeneous type\textquotedblright\ (in the sense of Coifman-Weiss \cite{CW}).
Instead, we have to set the problem in a weaker abstract context
(\textquotedblleft bounded nonhomogeneous spaces\textquotedblright), and apply
an ad hoc theory of singular integrals to get the desired $L^{p}$ bound. The
alluded ad hoc result has been proved by one of us in \cite{B}, in the spirit
of the theory of singular integrals in nonhomogeneous spaces, which has been
developed, since the late 1990's, by Nazarov-Treil-Volberg and other authors.
With this machinery at hand, we can prove the desired $L^{p}$ estimate for the
singular integral with kernel $\eta\cdot\partial_{x_{i}x_{j}}^{2}\Gamma$
\textit{on a ball}. To get the desired estimate on the whole strip
$\mathbb{R}^{N}\times\left[  -1,1\right]  $, still another nontrivial argument
is needed, based on a covering lemma and exploiting both the existence of a
group of translations, and the relevant properties of the quasidistance $d$.

\bigskip

\textbf{Acknowledgement. }\textit{We would like to thank Vincenzo Vespri for
raising a question that led us to the present study.}

\section{Background and known results\label{section background}}

\subsection*{\textbf{The principal part operator}}

Let us consider our operator $L,$ with the matrices $A,B$ written in the form
(\ref{A}), (\ref{B}). We denote by $B_{0}$ the matrix obtained by annihilating
every $\ast$ block in (\ref{B}):%
\begin{equation}
B_{0}=%
\begin{bmatrix}
0 & B_{1} & 0 & \ldots & 0\\
0 & 0 & B_{2} & \ldots & 0\\
\vdots & \vdots & \vdots & \ddots & \vdots\\
0 & 0 & 0 & \ldots & B_{r}\\
0 & 0 & 0 & \ldots & 0
\end{bmatrix}
\label{B_0}%
\end{equation}
with $B_{j}$ as in (\ref{B}). By \textit{principal part }of $L$ we mean the
operator%
\begin{equation}
L_{0}=div\left(  A\nabla\right)  +\left\langle x,B_{0}\nabla\right\rangle
-\partial_{t}. \label{L_0}%
\end{equation}
For any $\lambda>0,$ let us define the matrix of \textit{dilations on
}$\mathbb{R}^{N},$%
\[
D\left(  \lambda\right)  =diag\left(  \lambda I_{p_{0}},\lambda^{3}I_{p_{1}%
},...,\lambda^{2r+1}I_{p_{r}}\right)
\]
where $I_{p_{j}}$ denotes the $p_{j}\times p_{j}$ identity matrix, and the
matrix of \textit{dilations on }$\mathbb{R}^{N+1},$%
\[
\delta\left(  \lambda\right)  =diag\left(  \lambda I_{p_{0}},\lambda
^{3}I_{p_{1}},...,\lambda^{2r+1}I_{p_{r}},\lambda^{2}\right)  .
\]
Note that%
\[
\det\left(  \delta\left(  \lambda\right)  \right)  =\lambda^{Q+2}%
\]
where%
\[
Q+2=p_{0}+3p_{1}+...+\left(  2r+1\right)  p_{r}+2
\]
is called \textit{homogeneous dimension of} $\mathbb{R}^{N+1}.$ Analogously,%
\[
\det\left(  D\left(  \lambda\right)  \right)  =\lambda^{Q}%
\]
and $Q$ is called \textit{homogeneous dimension of} $\mathbb{R}^{N}.$ A
remarkable fact proved in \cite{LP} is that the operator $L_{0}$ is
homogeneous of degree two with respect to the dilations $\delta\left(
\lambda\right)  ,$ which by definition means that%
\[
L_{0}\left(  u\left(  \delta\left(  \lambda\right)  z\right)  \right)
=\lambda^{2}\left(  L_{0}u\right)  \left(  \delta\left(  \lambda\right)
z\right)  \text{ \ }%
\]
for any $u\in C_{0}^{\infty}\left(  \mathbb{R}^{N+1}\right)  ,z\in
\mathbb{R}^{N+1},\lambda>0.$

If we define%
\begin{equation}
C_{0}\left(  t\right)  =\int_{0}^{t}E_{0}\left(  s\right)  AE_{0}^{T}\left(
s\right)  ds\text{, where }E_{0}\left(  s\right)  =\exp\left(  -sB_{0}%
^{T}\right)  \label{C_0(t)}%
\end{equation}
then the operator $L_{0}$ turns out to be left invariant with respect to the
associated translations:%
\begin{align*}
\left(  x,t\right)  \odot\left(  \xi,\tau\right)   &  =\left(  \xi
+E_{0}\left(  \tau\right)  x,t+\tau\right)  ;\\
\left(  \xi,\tau\right)  ^{-1}  &  =\left(  -E_{0}\left(  -\tau\right)
\xi,-\tau\right)  \text{.}%
\end{align*}
Moreover, the dilations $z\mapsto\delta\left(  \lambda\right)  z$ are
automorphisms for the group $\left(  \mathbb{R}^{N+1},\odot\right)  $

There is a natural homogeneous norm in $\mathbb{R}^{N+1},$ induced by these
dilations:%
\[
\left\Vert \left(  x,t\right)  \right\Vert =\sum_{j=1}^{N}\left\vert
x_{j}\right\vert ^{1/q_{j}}+\left\vert t\right\vert ^{1/2}%
\]
where $q_{j}$ are positive integers such that $D\left(  \lambda\right)
=diag\left(  \lambda^{q_{1}},...,\lambda^{q_{N}}\right)  .$ Clearly, we have%
\[
\left\Vert \delta\left(  \lambda\right)  z\right\Vert =\lambda\left\Vert
z\right\Vert \text{ \ \ for any }\lambda>0,z\in\mathbb{R}^{N+1}.
\]
Other properties of $\left\Vert \cdot\right\Vert $ will be stated later.

\subsection*{\textbf{Fundamental solution}}

The following theorem collects some important known result about the
fundamental solution of $L$:

\begin{theorem}
Under the assumptions stated in the Introduction, the operator $L$ possesses a
fundamental solution%
\[
\Gamma\left(  z,\zeta\right)  =\gamma\left(  \zeta^{-1}\circ z\right)  \text{
for }z,\zeta\in\mathbb{R}^{N+1},
\]
with%
\begin{equation}
\gamma\left(  z\right)  =\left\{
\begin{array}
[c]{l}%
0\text{ for }t\leq0\\
\frac{\left(  4\pi\right)  ^{-N/2}}{\sqrt{\det C\left(  t\right)  }}%
\exp\left(  -\frac{1}{4}\left\langle C^{-1}\left(  t\right)  x,x\right\rangle
-t\text{Tr}B\right)  \text{ \ for }t>0
\end{array}
\right. \nonumber
\end{equation}
where $z=\left(  x,t\right)  $ and $C\left(  t\right)  $ is as in
(\ref{C(t)}). Recall that $C\left(  t\right)  $ is positive definite for any
$t>0$; hence $\gamma\in C^{\infty}\left(  \mathbb{R}^{N+1}\backslash\left\{
0\right\}  \right)  .$ The following representation formulas hold:%
\begin{align}
u\left(  z\right)   &  =-\left(  \gamma\ast Lu\right)  \left(  z\right)
=-\int_{\mathbb{R}^{N+1}}\gamma\left(  \zeta^{-1}\circ z\right)  Lu\left(
\zeta\right)  d\zeta;\label{repre formula 1}\\
\partial_{x_{i}x_{j}}^{2}u\left(  z\right)   &  =-PV\left(  \partial
_{x_{i}x_{j}}^{2}\gamma\ast Lu\right)  \left(  z\right)  +c_{ij}Lu\left(
z\right)  \label{repre formula 2}%
\end{align}
for any $u\in C_{0}^{\infty}\left(  \mathbb{R}^{N+1}\right)  ,$
$i,j=1,2,...,p_{0},$ for suitable constants $c_{ij}$ which we do not need to
specify. The \textquotedblleft principal value\textquotedblright\ in
(\ref{repre formula 2}) must be understood as%
\[
PV\left(  \partial_{x_{i}x_{j}}^{2}\gamma\ast Lu\right)  \left(  z\right)
\equiv\lim_{\varepsilon\rightarrow0}\int_{\left\Vert \zeta^{-1}\circ
z\right\Vert >\varepsilon}\left(  \partial_{x_{i}x_{j}}^{2}\gamma\right)
\left(  \zeta^{-1}\circ z\right)  Lu\left(  \zeta\right)  d\zeta.
\]

\end{theorem}

The above theorem is proved in \cite{H} (see also \cite{LP}), apart from
(\ref{repre formula 2}) which is proved in \cite[Proposition 2.11]{DP}.

The fundamental solution $\Gamma_{0}\left(  z,\zeta\right)  =\gamma_{0}\left(
\zeta^{-1}\circ z\right)  $ of the principal part operator $L_{0}$ enjoys
special properties; namely, for $t>0$%
\begin{equation}
\gamma_{0}\left(  x,t\right)  =\frac{\left(  4\pi\right)  ^{-N/2}}{\sqrt{\det
C_{0}\left(  t\right)  }}\exp\left(  -\frac{1}{4}\left\langle C_{0}%
^{-1}\left(  t\right)  x,x\right\rangle \right)  \label{gamma star}%
\end{equation}
with $C_{0}\left(  t\right)  $ as in (\ref{C_0(t)}); moreover (see
\cite[p.42]{LP}),
\begin{equation}
C_{0}\left(  \lambda^{2}t\right)  =D\left(  \lambda\right)  C_{0}\left(
t\right)  D\left(  \lambda\right)  \text{ \ \ \ }\forall\lambda,t>0
\label{C_0 dilat}%
\end{equation}
from which we can see that $\gamma_{0}$ is homogeneous of degree $-Q$:%
\[
\gamma_{0}\left(  \delta\left(  \lambda\right)  \left(  x,t\right)  \right)
=\lambda^{-Q}\gamma_{0}\left(  x,t\right)  \text{ \ \ }\forall\lambda
>0,\left(  x,t\right)  \in\mathbb{R}^{N+1}\setminus\left\{  \left(
0,0\right)  \right\}  .
\]
Furthermore, the following relation links $L$ to $L_{0}$ (see \cite[Lemma
3.3]{LP}):%
\begin{align}
\left\langle C\left(  t\right)  x,x\right\rangle  &  =\left\langle
C_{0}\left(  t\right)  x,x\right\rangle \left(  1+O\left(  t\right)  \right)
\text{ for }t\rightarrow0;\label{C-C_0 a}\\
\left\langle C^{-1}\left(  t\right)  x,x\right\rangle  &  =\left\langle
C_{0}^{-1}\left(  t\right)  x,x\right\rangle \left(  1+O\left(  t\right)
\right)  \text{ for }t\rightarrow0; \label{C-C_0 b}%
\end{align}
and (see \cite[eqt. (3.14)]{LP}):
\begin{equation}
\det C\left(  t\right)  =\det C_{0}\left(  t\right)  \left(  1+O\left(
t\right)  \right)  \text{ for }t\rightarrow0. \label{DetC-DetC_0}%
\end{equation}

\section{Estimate on the nonsingular part of the
integral\label{section nonsingular}}

We now localize the singular kernel appearing in (\ref{repre formula 2})
introducing a cutoff function
\begin{align*}
\eta &  \in C_{0}^{\infty}\left(  \mathbb{R}^{N+1}\right)  \text{ such that}\\
\eta\left(  z\right)   &  =1\text{ for }\left\Vert z\right\Vert \leq\rho
_{0}/2;\\
\eta\left(  z\right)   &  =0\text{ for }\left\Vert z\right\Vert \geq\rho_{0},
\end{align*}
where $\rho_{0}<1$ will be fixed later.

Let us rewrite (\ref{repre formula 2}) as:%
\begin{align}
\partial_{x_{i}x_{j}}^{2}u  &  =-PV\left(  \left(  \eta\partial_{x_{i}x_{j}%
}^{2}\gamma\right)  \ast Lu\right)  -\left(  \left(  1-\eta\right)
\partial_{x_{i}x_{j}}^{2}\gamma\ast Lu\right)  +c_{ij}Lu\label{repr formula}\\
&  \equiv-PV\left(  k_{0}\ast Lu\right)  -\left(  k_{\infty}\ast Lu\right)
+c_{ij}Lu\nonumber
\end{align}
having set:%
\begin{align}
k_{0}  &  =\eta\partial_{x_{i}x_{j}}^{2}\gamma\label{k_0}\\
k_{\infty}  &  =\left(  1-\eta\right)  \partial_{x_{i}x_{j}}^{2}%
\gamma\nonumber
\end{align}
for any $i,j=1,2,...,p_{0}$ (we will left implicitly understood the dependence
of the kernels $k_{0},k_{\infty}$ on these indices $i,j,$ as well as on the
number $\rho_{0}$ appearing in the definition of the cutoff function $\eta$).

Since in $k_{\infty}$ the singularity of $\partial_{x_{i}x_{j}}^{2}\gamma$ has
been removed and $\partial_{x_{i}x_{j}}^{2}\gamma$ has a fast decay as
$x\rightarrow\infty,$ we can prove the following:

\begin{proposition}
\label{Prop k_infty}For any $\rho_{0}>0\ $there exists \textbf{ }$c=c(\rho
_{0})>0$ such that for any $z\in S$%
\begin{align}
\int_{S}\left\vert k_{\infty}\left(  \zeta^{-1}\circ z\right)  \right\vert
d\zeta &  \leq c\label{k infty 1}\\
\int_{S}\left\vert k_{\infty}\left(  z^{-1}\circ\zeta\right)  \right\vert
d\zeta &  \leq c. \label{k infty 2}%
\end{align}

\end{proposition}

Note that this proposition immediately implies the following

\begin{corollary}
\label{Coroll nonsingular}For any $p\in\left[  1,\infty\right]  $ there exists
a constant $c>0$ only depending on $p,N,p_{0},\nu$ and the matrix $B$ such
that:%
\begin{equation}
\left\Vert -\left(  k_{\infty}\ast Lu\right)  +c_{ij}Lu\right\Vert
_{L^{p}\left(  S\right)  }\leq c\left\Vert Lu\right\Vert _{L^{p}\left(
S\right)  }\text{ for any }u\in C_{0}^{\infty}\left(  S\right)  ,
\label{bound k infty}%
\end{equation}
any $i,j=1,...,p_{0}.$
\end{corollary}

Before proving the proposition we need an easy lemma to handle a typical
change of variables in convolutions. It turns out that the Lebesgue measure is
invariant with respect to left translations, but not with respect to the
inversion $\zeta\longmapsto\zeta^{-1}$:

\begin{lemma}
\label{Lemma Jacobian}If we set $\zeta^{-1}\circ z=w=\left(  \xi,\tau\right)
,$ then the following identity holds for the Jacobian of the map
$w\longmapsto\zeta$ (for fixed $z$):%
\begin{equation}
d\zeta=e^{\tau\text{Tr}B}dw. \label{Jacob}%
\end{equation}

\end{lemma}

\begin{proof}
[Proof of the Lemma]Setting
\begin{align*}
\zeta^{-1}\circ z  &  =w;\\
\zeta &  =z\circ w^{-1},
\end{align*}
let us compute the Jacobian matrix of the map $w\longmapsto z\circ w^{-1}.$ If
$z=\left(  x,t\right)  ,w=\left(  \xi,\tau\right)  $ we have%
\[
z\circ w^{-1}=\left(  -E\left(  -\tau\right)  \xi+E\left(  -\tau\right)
x,t-\tau\right)  \text{,}%
\]
and the Jacobian is%
\[
J=%
\begin{bmatrix}
-E\left(  -\tau\right)  & \ast\\
0 & -1
\end{bmatrix}
\]
with determinant%
\[
\text{Det}J=\text{Det}\exp\left(  \tau B^{T}\right)  =e^{\tau\text{Tr}%
B}\text{.}%
\]

\end{proof}

\begin{proof}
[Proof of Proposition \ref{Prop k_infty}]Since we are not interested in the
exact dependence of the constant $c$ on $\rho_{0},$ for the sake of simplicity
we will prove the Proposition for $\rho_{0}=1$. An analogous proof can be done
for any $\rho_{0}$, finding a constant $c$ which depends on $\rho_{0}$.

Note that (\ref{k infty 2}) immediately follows by (\ref{k infty 1}), with the
change of variables $z^{-1}\circ\zeta=w^{-1}$ and applying the above Lemma.
So, let us prove (\ref{k infty 1}).

Recalling that, for $t>0,$ we have%
\[
\gamma\left(  x,t\right)  =\frac{\left(  4\pi\right)  ^{-N/2}}{\sqrt{\det
C\left(  t\right)  }}\exp\left(  -\frac{1}{4}\left\langle C^{-1}\left(
t\right)  x,x\right\rangle -t\text{Tr}B\right)  ,
\]
let us compute:%
\begin{align*}
\left(  \partial_{x_{i}}\gamma\right)  \left(  x,t\right)   &  =-\frac{1}%
{2}\gamma\left(  x,t\right)  \left\langle C^{-1}\left(  t\right)
x,e_{i}\right\rangle \\
\left(  \partial_{x_{i}x_{j}}^{2}\gamma\right)  \left(  x,t\right)   &
=\frac{1}{2}\gamma\left(  x,t\right)  \left\{  \frac{1}{2}\left\langle
C^{-1}\left(  t\right)  x,e_{j}\right\rangle \left\langle C^{-1}\left(
t\right)  x,e_{i}\right\rangle -\left\langle C^{-1}\left(  t\right)
e_{j},e_{i}\right\rangle \right\}
\end{align*}
(where we have denoted by $e_{i}$ the $i$-th unit vector in $\mathbb{R}^{N}$).
Since the matrix $C^{-1}\left(  t\right)  $ is symmetric and positive
definite, we can bound%
\[
\left\vert \left\langle C^{-1}\left(  t\right)  x,e_{j}\right\rangle
\right\vert \leq\left\langle C^{-1}\left(  t\right)  x,x\right\rangle
^{1/2}\left\langle C^{-1}\left(  t\right)  e_{j},e_{j}\right\rangle ^{1/2}.
\]
By (\ref{C-C_0 b}) and (\ref{C_0 dilat}) we have:%
\begin{equation}
\left\langle C^{-1}\left(  t\right)  e_{j},e_{j}\right\rangle =\left\langle
C_{0}^{-1}\left(  t\right)  e_{j},e_{j}\right\rangle \left(  1+O\left(
t\right)  \right)  \text{ \ \ for }t\rightarrow0 \label{C-C0}%
\end{equation}
and%
\begin{align*}
\left\langle C_{0}^{-1}\left(  t\right)  e_{j},e_{j}\right\rangle  &
=\left\langle C_{0}^{-1}\left(  1\right)  D\left(  \frac{1}{\sqrt{t}}\right)
e_{j},D\left(  \frac{1}{\sqrt{t}}\right)  e_{j}\right\rangle \leq c\left\vert
D\left(  \frac{1}{\sqrt{t}}\right)  e_{j}\right\vert ^{2}=\\
\text{(since }j  &  \in\left\{  1,2,...,p_{0}\right\}  \text{)}=c\left\vert
\frac{1}{\sqrt{t}}e_{j}\right\vert ^{2}=\frac{c}{t}.
\end{align*}
This shows that%
\begin{align*}
\left\langle C^{-1}\left(  t\right)  e_{j},e_{j}\right\rangle  &  \leq\frac
{c}{t}\left(  1+O\left(  t\right)  \right)  \text{, and}\\
\left\vert \left\langle C^{-1}\left(  t\right)  e_{j},e_{i}\right\rangle
\right\vert  &  \leq\left\langle C^{-1}\left(  t\right)  e_{j},e_{j}%
\right\rangle ^{1/2}\left\langle C^{-1}\left(  t\right)  e_{i},e_{i}%
\right\rangle ^{1/2}\leq\frac{c}{t}\left(  1+O\left(  t\right)  \right)  ,
\end{align*}
for $t\rightarrow0$. Therefore%
\begin{align*}
&  \left\vert \partial_{x_{i}x_{j}}^{2}\gamma\left(  x,t\right)  \right\vert
\leq\frac{1}{2}\gamma\left(  x,t\right)  \left\{  \frac{c}{t}\left\langle
C^{-1}\left(  t\right)  x,x\right\rangle +\frac{c}{t}\right\}  \left(
1+O\left(  t\right)  \right)  =\\
&  =\frac{c}{\sqrt{\det C\left(  t\right)  }}\exp\left(  -\frac{1}%
{4}\left\langle C^{-1}\left(  t\right)  x,x\right\rangle -t\text{Tr}B\right)
\left\{  \frac{1}{t}\left\langle C^{-1}\left(  t\right)  x,x\right\rangle
+\frac{1}{t}\right\}  \left(  1+O\left(  t\right)  \right) \\
&  \leq\frac{c}{t\sqrt{\det C\left(  t\right)  }}\exp\left(  -\frac{\left(
1-\delta\right)  }{4}\left\langle C^{-1}\left(  t\right)  x,x\right\rangle
-t\text{Tr}B\right)
\end{align*}
for some $\delta>0,$ any $t\in\left[  -1,1\right]  $, since%
\[
\exp\left(  -\frac{1}{4}\left\langle C^{-1}\left(  t\right)  x,x\right\rangle
\right)  \left\{  \left\langle C^{-1}\left(  t\right)  x,x\right\rangle
+1\right\}  \leq c\exp\left(  -\frac{\left(  1-\delta\right)  }{4}\left\langle
C^{-1}\left(  t\right)  x,x\right\rangle \right)  ,
\]
which follows from%
\begin{align*}
\left(  4\alpha+1\right)  \exp\left(  -\alpha\right)   &  \leq c\exp\left(
-\left(  1-\delta\right)  \alpha\right)  \text{ \ \ \ }\forall\alpha\geq0\\
\text{and }\alpha &  =\frac{1}{4}\left\langle C^{-1}\left(  t\right)
x,x\right\rangle \geq0.
\end{align*}
Let us rewrite the last inequality as%
\begin{align}
\left\vert \partial_{x_{i}x_{j}}^{2}\gamma\left(  x,t\right)  \right\vert  &
\leq\frac{c}{t}\gamma_{\delta}\left(  x,t\right)  ,\label{bound D2 gamma}\\
\text{with }\gamma_{\delta}\left(  x,t\right)   &  =\frac{\left(  4\pi\right)
^{-N/2}}{\sqrt{\det C\left(  t\right)  }}\exp\left(  -\frac{\left(
1-\delta\right)  }{4}\left\langle C^{-1}\left(  t\right)  x,x\right\rangle
-t\text{Tr}B\right)  .\nonumber
\end{align}

With this bound in hand, we can now evaluate the following integral, for $z\in
S$. Using Lemma \ref{Lemma Jacobian} we have:%
\begin{align*}
&  \int_{S}\left\vert \left(  \left(  1-\eta\right)  \partial_{x_{i}x_{j}}%
^{2}\gamma\right)  \left(  \zeta^{-1}\circ z\right)  \right\vert d\zeta\leq\\
&  \leq c\int_{\mathbb{R}^{N}\times\left(  -2,2\right)  ,\left\Vert
\zeta^{\prime}\right\Vert \geq1/2}\left\vert \left(  \left(  1-\eta\right)
\partial_{x_{i}x_{j}}^{2}\gamma\right)  \left(  \zeta^{\prime}\right)
\right\vert d\zeta^{\prime}\\
&  =c\int_{\mathbb{R}^{N}\times\left(  -2,2\right)  ,\left\Vert \zeta^{\prime
}\right\Vert \geq1/2,\left\Vert \left(  x,0\right)  \right\Vert \leq
1/4}\left\vert \left(  \left(  1-\eta\right)  \partial_{x_{i}x_{j}}^{2}%
\gamma\right)  \left(  x,t\right)  \right\vert dxdt+\\
&  +c\int_{\mathbb{R}^{N}\times\left(  -2,2\right)  ,\left\Vert \zeta^{\prime
}\right\Vert \geq1/2,\left\Vert \left(  x,0\right)  \right\Vert >1/4}%
\left\vert \left(  \left(  1-\eta\right)  \partial_{x_{i}x_{j}}^{2}%
\gamma\right)  \left(  x,t\right)  \right\vert dxdt\\
&  \equiv I+II.
\end{align*}
Now,%
\begin{align*}
I  &  \leq c\int_{1/16\leq\left\vert t\right\vert \leq2,\left\Vert \left(
x,0\right)  \right\Vert \leq1/4}\frac{\left(  4\pi\right)  ^{-N/2}}%
{t\sqrt{\det C\left(  t\right)  }}\exp\left(  -\frac{\left(  1-\delta\right)
}{4}\left\langle C^{-1}\left(  t\right)  x,x\right\rangle -t\text{Tr}B\right)
dxdt\\
&  \leq c\int_{\left\vert x\right\vert \leq c_{1}}\exp\left(  -c_{2}\left\vert
x\right\vert ^{2}\right)  dx\leq c
\end{align*}
where we used (\ref{C-C_0 b}) and the fact that%
\[
\left\langle C_{0}^{-1}\left(  t\right)  x,x\right\rangle \geq c\left\vert
D\left(  \frac{1}{\sqrt{t}}\right)  x\right\vert ^{2}\geq c\left\vert
x\right\vert ^{2}\text{ since }\left\vert t\right\vert \leq2
\]
while, by (\ref{DetC-DetC_0}),%
\[
t\sqrt{\det C\left(  t\right)  }\geq c_{1}\sqrt{\det C_{0}\left(  t\right)
}=c_{2}t^{\left(  Q+2\right)  /2}\geq c_{3}\text{ since }\left\vert
t\right\vert \geq1/16.
\]
To handle $II$, we start noting that, if $\left\Vert \left(  x,0\right)
\right\Vert >1/4,$ by (\ref{C-C_0 b}) we can write%
\begin{align*}
&  \exp\left(  -c_{1}\left\langle C^{-1}\left(  t\right)  x,x\right\rangle
\right)  \leq\exp\left(  -c_{2}\left\langle C_{0}^{-1}\left(  t\right)
x,x\right\rangle \right)  \leq\\
&  \leq\exp\left(  -c_{3}\left\vert D\left(  \frac{1}{\sqrt{t}}\right)
x\right\vert ^{2}\right)  \leq\exp\left(  -c_{4}\frac{\left\vert x\right\vert
^{2}}{t}\right)  \leq\\
&  \leq\exp\left(  -\frac{c_{5}}{t}\right)  \leq c_{6}t
\end{align*}
hence%
\begin{align*}
II  &  \leq\int_{\mathbb{R}^{N}\times\left(  -2,2\right)  ,\left\Vert
\zeta^{\prime}\right\Vert \geq1/2,\left\Vert \left(  x,0\right)  \right\Vert
>1/4}\frac{c}{\sqrt{\det C\left(  t\right)  }}\exp\left(  -c_{7}\left\langle
C^{-1}\left(  t\right)  x,x\right\rangle \right)  dxdt=\\
&  \text{(letting }C^{-1/2}\left(  t\right)  x=y\text{)}\\
&  =\int_{\mathbb{R}^{N}\times\left(  -2,2\right)  }c\exp\left(
-c_{7}\left\vert y\right\vert ^{2}\right)  dydt=c.
\end{align*}

\end{proof}

By Corollary \ref{Coroll nonsingular} and (\ref{repr formula}), our final goal
will be achieved as soon as we will prove that%
\begin{equation}
\left\Vert PV\left(  k_{0}\ast Lu\right)  \right\Vert _{L^{p}\left(  S\right)
}\leq c\left\Vert Lu\right\Vert _{L^{p}\left(  S\right)  }
\label{bound gamma_0}%
\end{equation}
for any $u\in C_{0}^{\infty}\left(  S\right)  ,$ $i,j=1,...,p_{0},1<p<\infty.$
The proof of (\ref{bound gamma_0}) will be carried out in the following
sections, and concluded with Theorem \ref{Thm conclusion}.

\section{Estimates on the singular kernel}

To prove the singular integral estimate (\ref{bound gamma_0}), we have to
introduce some more structure in our setting. Let:%
\[
d\left(  z,\zeta\right)  =\left\Vert \zeta^{-1}\circ z\right\Vert .
\]
Recall that $\circ$ is the translation induced by the the operator $L$ (or
more precisely by the matrix $B$), and $\left\Vert \cdot\right\Vert $ the
homogeneous norm induced by the dilations associated to the principal part
operator $L_{0}$ (see \S \ref{section background}). This object has been
introduced and used in \cite{DP}, and turns out to be the right geometric tool
to describe the properties of the singular kernel $\gamma_{0}$. Namely, the
following key properties have been proved in \cite{DP}:

\begin{proposition}
\label{Prop quasidist}(See Lemma 2.1 in \cite{DP}). For any compact set
$K\subset\mathbb{R}^{N}$ there exists a constant $c_{K}\geq1$ such that
\begin{align*}
\left\Vert z^{-1}\right\Vert  &  \leq c_{K}\left\Vert z\right\Vert \text{ for
every }z\in K\times\left[  -1,1\right] \\
\left\Vert z\circ\zeta\right\Vert  &  \leq c_{K}\left\{  \left\Vert
z\right\Vert +\left\Vert \zeta\right\Vert \right\}  \text{ for every }\zeta\in
S,z\in K\times\left[  -1,1\right]  .
\end{align*}
In terms of $d,$ the above inequalities imply the following:%
\begin{align*}
d\left(  z,\zeta\right)   &  \leq cd\left(  \zeta,z\right)  \text{ }\forall
z,\zeta\in S\text{ with }d\left(  \zeta,z\right)  \leq1\\
d\left(  z,\zeta\right)   &  \leq c\left\{  d\left(  z,w\right)  +d\left(
w,\zeta\right)  \right\}  \text{ }\forall z,\zeta,w\in S\text{ with }d\left(
z,w\right)  \leq1,d\left(  w,\zeta\right)  \leq1.
\end{align*}

\end{proposition}

Let us define the $d$-balls:%
\[
B\left(  z,\rho\right)  =\left\{  \zeta\in\mathbb{R}^{N+1}:d\left(
z,\zeta\right)  <\rho\right\}  .
\]

\begin{lemma}
\label{Lemma topologies}The $d$-balls are open with respect to the Euclidean
topology. Moreover, the topology induced by this family of balls (saying that
a set $\Omega$ is open whenever for any $x\in\Omega$ there exists $\rho>0$
such that $B\left(  x,\rho\right)  \subset\Omega$) coincides with the
Euclidean topology.
\end{lemma}

\begin{proof}
Since the function $\zeta\mapsto d\left(  z_{0},\zeta\right)  $ is continuous,
$B\left(  z_{0},\rho\right)  $ is open with respect to the Euclidean topology;
in particular, $B\left(  z_{0},\rho\right)  $ contains an Euclidean ball
centered at $z_{0}.$

Conversely, fix an Euclidean ball $B^{E}\left(  z_{0},\rho\right)  $ of center
$z_{0}$ and radius $\rho>0,$ and assume that $z$ is point such that%
\[
\left\Vert z^{-1}\circ z_{0}\right\Vert <\varepsilon,
\]
for some $\varepsilon>0$ to be chosen later. Then, letting $w=z^{-1}\circ
z_{0},$ we have%
\[
\left\vert z_{0}-z\right\vert =\left\vert z_{0}-z_{0}\circ w^{-1}\right\vert
<\rho\text{ for }\varepsilon\text{ small enough,}%
\]
because
\[
\left\vert w^{-1}\right\vert \leq c\left\Vert w^{-1}\right\Vert \leq
c\left\Vert w\right\Vert <c\varepsilon,
\]
and the translation $\circ$ is a smooth operation. Hence
\[
B^{E}\left(  z_{0},\rho\right)  \supseteq B\left(  z_{0},\varepsilon\right)
,
\]
so that the two topologies coincide.
\end{proof}

The relevant information about the measure of $d$-balls are contained in the following:

\begin{proposition}
\label{Prop measure balls}

\begin{enumerate}
\item[(i)] The following dimensional bound holds:%
\[
\left\vert B\left(  z,\rho\right)  \right\vert \leq c\rho^{Q+2}\text{ for any
}z\in S,0<\rho<1.
\]

\item[(ii)] The following doubling condition holds in $S$:%
\[
\left\vert B\left(  z,2\rho\right)  \cap S\right\vert \leq c\left\vert
B\left(  z,\rho\right)  \cap S\right\vert \text{ \ for any }z\in S,0<\rho<1.
\]

\end{enumerate}
\end{proposition}

\begin{proof}
Let us compute the integral%
\[
\left\vert B\left(  z,\rho\right)  \right\vert =\int_{\left\Vert \zeta
^{-1}\circ z\right\Vert <\rho}d\zeta.
\]
Setting $\zeta^{-1}\circ z=w$ and applying Lemma \ref{Lemma Jacobian} we have,
if $z=\left(  x,t\right)  ,w=\left(  \xi,\tau\right)  $:
\[
\left\vert B\left(  z,\rho\right)  \right\vert =\int_{\left\Vert \left(
\xi,\tau\right)  \right\Vert <\rho}e^{\tau\text{Tr}B}d\xi d\tau
\]
Since $z\in S,$ in particular, $\left\vert t\right\vert \leq1,$ $\left\vert
t-\tau\right\vert \leq\rho^{2},$ hence $\left\vert \tau\right\vert \leq2$ and
the last integral is%
\[
\leq e^{2\text{Tr}B}\int_{\left\Vert w\right\Vert <\rho}dw
\]
by the dilation $w=\delta\left(  \rho\right)  w^{\prime}$%
\[
=e^{2\text{Tr}B}\rho^{Q+2}\int_{\left\Vert \left(  \xi,\tau\right)
\right\Vert <1}d\xi d\tau=c\rho^{Q+2}%
\]
which proves (i).

To prove (ii), let $\zeta=\left(  x^{\prime},t^{\prime}\right)  ,z=\left(
x,t\right)  ,w=\left(  \xi,\tau\right)  ,$ and assume, to fix ideas, $t\geq0.$
Then%
\begin{align*}
\left\vert B\left(  z,\rho\right)  \cap S\right\vert  &  =\int_{\left\Vert
\left(  \xi,\tau\right)  \right\Vert <\rho,\left\vert t-\tau\right\vert
<1}e^{\tau\text{Tr}B}d\xi d\tau\\
&  \geq c\int_{\left\Vert \left(  \xi,\tau\right)  \right\Vert <\rho,0\leq
\tau\leq1}d\xi d\tau\text{ (since }\rho<1\text{)}\\
&  =\frac{c}{2}\int_{\left\Vert w\right\Vert <\rho}dw=c\rho^{Q+2}%
\int_{\left\Vert w^{\prime}\right\Vert <1}dw^{\prime}=c\rho^{Q+2}\\
&  \geq c\left\vert B\left(  z,2\rho\right)  \right\vert \geq c\left\vert
B\left(  z,2\rho\right)  \cap S\right\vert
\end{align*}
by (i).
\end{proof}

We also need the following bounds of the fundamental solution $\Gamma$ in
terms of $d$:

\begin{proposition}
(See Proposition 2.7 in \cite{DP}) The following \textquotedblleft standard
estimates\textquotedblright\ hold for $\Gamma$ in terms of $d$: there exist
$c>0$ and $M>1$ such that%
\begin{align*}
\left\vert \partial_{x_{i}x_{j}}^{2}\Gamma\left(  z,\zeta\right)  \right\vert
&  \leq\frac{c}{d\left(  z,\zeta\right)  ^{Q+2}}\text{ \ \ }\forall z,\zeta\in
S\\
\left\vert \partial_{x_{i}x_{j}}^{2}\Gamma\left(  \zeta,w\right)
-\partial_{x_{i}x_{j}}^{2}\Gamma\left(  z,w\right)  \right\vert  &  \leq
c\frac{d\left(  w,z\right)  }{d\left(  w,\zeta\right)  ^{Q+3}}\text{
\ \ }\forall z,\zeta,w\in S\text{ }%
\end{align*}
with $Md\left(  w,z\right)  \leq d\left(  w,\zeta\right)  \leq1.$
\end{proposition}

An easy computation shows that the previous estimates extend to the kernel
$k_{0}=\eta\partial_{x_{i}x_{j}}^{2}\gamma$:

\begin{proposition}
\label{Prop kernel}There exists $c>0$ and $M>1$ such that%
\begin{align*}
\left\vert k_{0}\left(  \zeta^{-1}\circ z\right)  \right\vert  &  \leq\frac
{c}{d\left(  z,\zeta\right)  ^{Q+2}}\text{ \ \ }\forall z,\zeta\in S\\
\left\vert k_{0}\left(  w^{-1}\circ\zeta\right)  -k_{0}\left(  w^{-1}\circ
z\right)  \right\vert  &  \leq c\frac{d\left(  w,z\right)  }{d\left(
w,\zeta\right)  ^{Q+3}}\text{ \ \ }\forall z,\zeta,w\in S\text{ }%
\end{align*}
with $Md\left(  w,z\right)  \leq d\left(  w,\zeta\right)  \leq1.$
\end{proposition}

\begin{remark}
\label{Remark M}We can always assume that $M$ is large enough, so that the
conditions%
\[
Md\left(  w,z\right)  \leq d\left(  w,\zeta\right)  \leq1
\]
imply%
\[
c_{1}d\left(  z,\zeta\right)  \leq d\left(  w,\zeta\right)  \leq c_{2}d\left(
z,\zeta\right)
\]
for some absolute constants $c_{1},c_{2}>0.$
\end{remark}

We will also need the following:

\begin{lemma}
\label{Lemma cancellation}There exists $c>0$ such that%
\[
\left\vert \int_{r_{1}<\left\Vert \zeta^{-1}\circ z\right\Vert <r_{2}}%
k_{0}\left(  \zeta^{-1}\circ z\right)  d\zeta\right\vert \leq c\text{ }%
\]
for any $z\in S,0<r_{1}<r_{2}.$ Moreover, for every $z\in S$, the limit
\[
\lim_{\varepsilon\rightarrow0^{+}}\int_{\left\Vert \zeta^{-1}\circ
z\right\Vert >\varepsilon}k_{0}\left(  \zeta^{-1}\circ z\right)  d\zeta
\]
exists, is finite, and independent of $z$.
\end{lemma}

\begin{proof}
The change of variables $w=\zeta^{-1}\circ z$ (see Lemma \ref{Lemma Jacobian})
shows that%
\begin{align}
\int_{r_{1}<\left\Vert \zeta^{-1}\circ z\right\Vert <r_{2}}k_{0}\left(
\zeta^{-1}\circ z\right)  d\zeta &  =\int_{r_{1}<\left\Vert w\right\Vert
<r_{2}}k_{0}\left(  w\right)  e^{\tau\text{Tr}B}dw=\label{shell}\\
\text{ (for }r_{2}  &  \leq\frac{\rho_{0}}{2}\text{) \ \ }=\int_{r_{1}%
<\left\Vert w\right\Vert <r_{2}}\partial_{x_{i}x_{j}}^{2}\gamma\left(
w\right)  e^{\tau\text{Tr}B}dw\nonumber
\end{align}
with $w=\left(  \xi,\tau\right)  $. \ However, by the divergence theorem the
last integral equals%
\begin{align*}
&  \int_{\left\Vert w\right\Vert =r_{2}}\partial_{x_{i}}\gamma\left(
w\right)  e^{\tau\text{Tr}B}\nu_{j}d\sigma\left(  w\right)  -\int_{\left\Vert
w\right\Vert =r_{1}}\partial_{x_{i}}\gamma\left(  w\right)  e^{\tau\text{Tr}%
B}\nu_{j}d\sigma\left(  w\right) \\
&  \equiv I\left(  r_{2}\right)  -I\left(  r_{1}\right)  .
\end{align*}
It is shown in \cite[Lemma 2.10]{LP} that%
\[
I\left(  \rho\right)  \rightarrow\int_{\left\Vert w\right\Vert =1}%
\partial_{x_{i}}\gamma_{0}\left(  w\right)  e^{\tau\text{Tr}B}\nu_{j}%
d\sigma\left(  w\right)  \text{ as }\rho\rightarrow0
\]
with $\gamma_{0}$ as in (\ref{gamma star}). Since, on the other hand,
$I\left(  \rho\right)  $ is continuous for $\rho\in(0,1/2],$ we conclude that
$I\left(  \rho\right)  $ is bounded for $\rho\in\left[  0,\frac{\rho_{0}}%
{2}\right]  $. This implies the first statement in the Lemma if $r_{2}\leq
\rho_{0}/2.$ Note that we can always assume $r_{2}\leq\rho_{0}$, because
$k_{0}\left(  w\right)  =0$ for $\left\Vert w\right\Vert >\rho_{0}.$ Then, if
$\rho_{0}/2\leq r_{2}\leq\rho_{0},$ we can write%
\[
\left\vert \int_{\rho_{0}/2\leq\left\Vert w\right\Vert <r_{2}}k_{0}\left(
w\right)  e^{\tau\text{Tr}B}dw\right\vert \leq\int_{\rho_{0}/2\leq\left\Vert
w\right\Vert \leq\rho_{0}}c\left\Vert w\right\Vert ^{-\left(  2+Q\right)
}dw=c.
\]

The second statement follows by a similar argument.
\end{proof}

\section{$L^{p}$ estimates on singular integrals on nonhomogeneous spaces}

We now want to apply to our singular kernel an abstract result, proved in
\cite{B}, which we are going to recall now.

Let $X$ be a set. A function $d:X\times X\rightarrow\mathbb{R}$ is called a
\textit{quasisymmetric quasidistance }on $X$ if there exists a \ constant
$c_{d}\geqslant1$ such that for any $x,y,z\in X$:
\[
d\left(  x,y\right)  \geqslant0\text{ and }d\left(  x,y\right)
=0\Leftrightarrow x=y;
\]%
\begin{equation}
d\left(  x,y\right)  \leq c_{d}d\left(  y,x\right)  ; \label{sym}%
\end{equation}%
\begin{equation}
d\left(  x,y\right)  \leqslant c_{d}\left(  d\left(  x,z\right)  +d\left(
z,y\right)  \right)  . \label{triang}%
\end{equation}
If $d$ is a quasisymmetric quasidistance, then%
\[
d^{\ast}\left(  x,y\right)  =d\left(  x,y\right)  +d\left(  y,x\right)
\]
is a quasidistance, equivalent to $d$; $d^{\ast}$ will be called \textit{the
symmetrized quasidistance }of $d.$

\begin{definition}
\label{Def nonhomogeneous}We will say that $\left(  X,d,\mu,k\right)  $ is a
\emph{nonhomogeneous space with Calder\'{o}n-Zygmund kernel} $k$ if:

\begin{enumerate}
\item $\left(  X,d\right)  $ is a set endowed with a quasisymmetric
quasidistance $d,$ such that the $d$-balls are open with respect to the
topology induced by $d$;

\item $\mu$ is a positive regular Borel measure on $X,$ and there exist two
positive constants $A,n$ such that:%
\begin{equation}
\mu\left(  B\left(  x,\rho\right)  \right)  \leq A\rho^{n}\text{ for any }x\in
X,\rho>0;\label{dimension}%
\end{equation}

\item $k\left(  x,y\right)  $ is a real valued measurable kernel defined in
$X\times X,$ and there exists a positive constant $\beta$ such that:%
\begin{equation}
\left\vert k\left(  x,y\right)  \right\vert \leq\frac{A}{d\left(  x,y\right)
^{n}}\text{ for any }x,y\in X; \label{growth}%
\end{equation}%
\begin{equation}
\left\vert k\left(  x,y\right)  -k\left(  x_{0},y\right)  \right\vert \leq
A\frac{d\left(  x_{0},x\right)  ^{\beta}}{d\left(  x_{0},y\right)  ^{n+\beta}}
\label{mean value}%
\end{equation}
for any $x_{0},x,y\in X$ with $d\left(  x_{0},y\right)  \geq Ad\left(
x_{0},x\right)  ,$ where $n,A$ are as in (\ref{dimension}).
\end{enumerate}
\end{definition}

\begin{theorem}
\label{Thm Lp}(See Theorem 3 in \cite{B}). Let $\left(  X,d,\mu,k\right)  $ be
a bounded and separable nonhomogeneous space with Calder\'{o}n-Zygmund kernel
$k.$ Also, assume that

(i) $k^{\ast}\left(  x,y\right)  \equiv k\left(  y,x\right)  $ satisfies
(\ref{mean value});

(ii) there exists a constant $B>0$ such that%
\begin{equation}
\left\vert \int_{d\left(  x,y\right)  >\rho}k\left(  x,y\right)  d\mu\left(
y\right)  \right\vert +\left\vert \int_{d\left(  x,y\right)  >\rho}k^{\ast
}\left(  x,y\right)  d\mu\left(  y\right)  \right\vert \leqslant B
\label{cancellation}%
\end{equation}
for any $\rho>0,x\in X$;

(iii) for a.e. $x\in X,$ the limits
\[
\lim_{\rho\rightarrow0}\int_{d\left(  x,y\right)  >\rho}k\left(  x,y\right)
d\mu\left(  y\right)  ;\text{ \ \ }\lim_{\rho\rightarrow0}\int_{d\left(
x,y\right)  >\rho}k^{\ast}\left(  x,y\right)  d\mu\left(  y\right)
\]
exist finite. Then the operator%
\[
Tf\left(  x\right)  \equiv\lim_{\varepsilon\rightarrow0}T_{\varepsilon
}f\left(  x\right)  \equiv\lim_{\varepsilon\rightarrow0}\int_{d\left(
x,y\right)  >\varepsilon}k\left(  x,y\right)  f\left(  y\right)  d\mu\left(
y\right)
\]
is well defined for any $f\in L^{1}\left(  X\right)  ,$ and%
\[
\left\Vert Tf\right\Vert _{L^{p}\left(  X\right)  }\leq c_{p}\left\Vert
f\right\Vert _{L^{p}\left(  X\right)  }\text{ for any }p\in\left(
1,\infty\right)  ;
\]
moreover, $T$ is weakly $\left(  1,1\right)  $ continuous. The constant
$c_{p}$ only depends on all the constants implicitly involved in the
assumptions: $p,c_{d},A,B,n,\beta,$diam$\left(  X\right)  $.
\end{theorem}

We will also need the notion of H\"{o}lder space in this context:

\begin{definition}
[H\"{o}lder spaces]We will say that $f\in C^{\alpha}\left(  X\right)  $, for
some $\alpha>0,$ if%
\[
\left\Vert f\right\Vert _{\alpha}\equiv\left\Vert f\right\Vert _{\infty
}+\left\vert f\right\vert _{\alpha}\equiv\sup_{x\in X}\left\vert f\left(
x\right)  \right\vert +\sup_{x,y\in X,x\neq y}\frac{\left\vert f\left(
x\right)  -f\left(  y\right)  \right\vert }{d\left(  x,y\right)  ^{\alpha}%
}<\infty.
\]

\end{definition}

Our aim now is to apply the previous abstract result to the singular integral
$T$ with kernel $k_{0}$ on a bounded domain, say a ball $B\left(
z_{0},R\right)  .$ More precisely, as we shall see later, what we need is an
estimate of the kind%
\[
\left\Vert Tf\right\Vert _{L^{p}\left(  B\left(  z_{0},R\right)  \right)
}\leq c\left\Vert f\right\Vert _{L^{p}\left(  B\left(  z_{0},R\right)
\right)  }%
\]
for $1<p<\infty,$ where $R$ is a small radius fixed once and for all, $z_{0}$
is any point in the strip $S,$ and the constant $c$ is independent from
$z_{0}$. \ Note that, by Proposition \ref{Prop quasidist}, our $d$ is actually
a quasisymmetric quasidistance in $X=B\left(  z_{0},R\right)  ,$ as soon as
$R$ is small enough; moreover, by Proposition \ref{Prop measure balls} the
Lebesgue measure of a $d$-ball satisfies the required dimensional bound
(\ref{dimension}) with $n=Q+2$. Also, Proposition \ref{Prop kernel} and Lemma
\ref{Lemma cancellation} suggest that the kernel $k_{0}$ satisfies the
properties required by Theorem \ref{Thm Lp}. However, there is a subtle
problem with this last assertion. Namely, saying, for instance, that $k_{0}$
satisfies the cancellation property in $B\left(  z_{0},R\right)  $ means that%
\[
\left\vert \int_{\zeta\in B\left(  z_{0},R\right)  :r_{1}<d\left(
z,\zeta\right)  <r_{2}}k_{0}\left(  \zeta^{-1}\circ z\right)  d\zeta
\right\vert \leq c
\]
whereas what we know (see Lemma \ref{Lemma cancellation}) is that%
\[
\left\vert \int_{\zeta\in\mathbb{R}^{N+1}:r_{1}<d\left(  z,\zeta\right)
<r_{2}}k_{0}\left(  \zeta^{-1}\circ z\right)  d\zeta\right\vert \leq c.
\]
The problem is that restricting the kernel $k_{0}$ to the domain $B\left(
z_{0},R\right)  $ has the effect of a rough cut on the kernel, which can harm
the validity of the cancellation property. To realize how things can actually
go wrong, take the restriction of the Hilbert transform on the interval
$\left(  0,1\right)  $: the singular integral operator%
\[
Tf\left(  x\right)  =\lim_{\varepsilon\rightarrow0}\int_{y\in\left(
0,1\right)  ,\left\vert x-y\right\vert >\varepsilon}\frac{f\left(  y\right)
}{x-y}dy
\]
is not so friendly, because
\[
T1\left(  x\right)  =\lim_{\varepsilon\rightarrow0}\int_{y\in\left(
0,1\right)  ,\left\vert x-y\right\vert >\varepsilon}\frac{1}{x-y}%
dy=\log\left(  \frac{x}{1-x}\right)  \text{ for any }x\in\left(  0,1\right)
\]
so (\ref{cancellation}) does not hold in this case. A more cautious choice,
then, consists in cutting the kernel smoothly, by a couple of H\"{o}lder
continuous cutoff functions. Namely, we have the following

\begin{proposition}
\label{Prop check assumptions}Let $k_{0}$ be the above kernel (see
(\ref{k_0})). There exists a constant $R_{0}>0$\ such that, for any $z_{0}\in
S,$ $R\leq R_{0},$ if $a,b$ are two cutoff functions belonging to $C^{\alpha
}\left(  \mathbb{R}^{N+1}\right)  $ for some $\alpha>0,$ with sprt$\,a,$
sprt$\,b\subset B\left(  z_{0},R\right)  ,$ and we set%
\begin{equation}
k\left(  x,y\right)  =a\left(  x\right)  k_{0}\left(  y^{-1}\circ x\right)
b\left(  y\right)  , \label{k}%
\end{equation}
then:

(a) $k$ satisfies (\ref{growth}),(\ref{mean value}) and (\ref{cancellation})
in $B\left(  z_{0},R\right)  $ (with possibly other constants). Explicitly,
\textquotedblleft(\ref{cancellation}) in $B\left(  z_{0},R\right)
$\textquotedblright\ means%
\begin{equation}
\left\vert \int_{y\in B\left(  z_{0},R\right)  :r_{1}<d\left(  x,y\right)
<r_{2}}k\left(  x,y\right)  d\mu\left(  y\right)  \right\vert \leq c.
\label{canc}%
\end{equation}

(b) for any $x\in B\left(  z_{0},R\right)  $ there exists%
\[
h\left(  x\right)  \equiv\lim_{\varepsilon\rightarrow0}\int_{y\in B\left(
z_{0},R\right)  :d\left(  x,y\right)  >\varepsilon}k\left(  x,y\right)
d\mu\left(  y\right)  .
\]
Finally, all the constants appearing in the above estimates about $k$ depend
on $z_{0},R$ and the cutoff functions $a,b$ only through the $C^{\alpha}$
norms of $a,b$.
\end{proposition}

\begin{remark}
Since in this Proposition and its proof the distinction between space and time
variables is irrelevant, changing for a moment our notation we have denoted by
$x,y,x_{0}...$ the variables in $\mathbb{R}^{N+1},$ and by $d\mu$ the Lebesgue
measure $dxdt$ in $\mathbb{R}^{N+1}.$
\end{remark}

\begin{proof}
We will apply several times the properties of the kernel $k_{0}$ proved in
Proposition \ref{Prop kernel} and Lemma \ref{Lemma cancellation}. Also, we
will use twice the following simple fact:%
\begin{equation}
\int_{d\left(  x,y\right)  <\rho}\frac{d\mu\left(  y\right)  }{d\left(
x,y\right)  ^{Q+2-\alpha}}\leq c\rho^{\alpha}\text{ for any }\rho>0
\label{hom}%
\end{equation}
which can be checked by a dilation argument.

We chose $R_{0}$ small enough so that $x,y\in B\left(  z_{0},R_{0}\right)  $
imply
\[
d\left(  x,y\right)  +d\left(  y,x\right)  \leq1\text{.}%
\]
Let $0<R\leq R_{0}.$

(a). Condition (\ref{growth}) for $k$ in $B\left(  z_{0},R\right)  $ obviously
follows from the analogous property of $k_{0}$. As to (\ref{mean value}), we
can write%
\begin{align*}
k\left(  x,y\right)  -k\left(  x_{0},y\right)   &  =\left[  a\left(  x\right)
-a\left(  x_{0}\right)  \right]  k_{0}\left(  y^{-1}\circ x\right)  b\left(
y\right)  +\\
+a\left(  x_{0}\right)  \left[  k_{0}\left(  y^{-1}\circ x\right)
-k_{0}\left(  y^{-1}\circ x_{0}\right)  \right]  b\left(  y\right)   &  =I+II.
\end{align*}
Now, for $d\left(  x_{0},y\right)  >Md\left(  x_{0},x\right)  $%
\[
\left\vert I\right\vert \leq\left\vert a\right\vert _{\alpha}d\left(
x,x_{0}\right)  ^{\alpha}\cdot\frac{c}{d\left(  x,y\right)  ^{Q+2}}\left\Vert
b\right\Vert _{\infty}\leq c\frac{d\left(  x,x_{0}\right)  ^{\alpha}}{d\left(
x_{0},y\right)  ^{Q+2+\alpha}}.
\]
We have implicitly used the fact that the functions $d\left(  x_{0},y\right)
,d\left(  x,y\right)  $ are bounded by some absolute constant (since
$x_{0},x,y\in B\left(  z_{0},R\right)  $), and the equivalence between
$d\left(  x_{0},y\right)  $ and $d\left(  x,y\right)  ,$ which holds under the
assumption $d\left(  x_{0},y\right)  >Md\left(  x_{0},x\right)  $ (see Remark
\ref{Remark M}).

Moreover, since $k_{0}$ satisfies (\ref{mean value}),
\[
\left\vert II\right\vert \leq\left\Vert a\right\Vert _{\infty}\left\Vert
b\right\Vert _{\infty}c\frac{d\left(  x,x_{0}\right)  }{d\left(
x_{0},y\right)  ^{Q+3}},
\]
hence (\ref{mean value}) holds for $k$ in $B\left(  z_{0},R\right)  ,$ with
$n=Q+2,\beta=\alpha$.

To check (\ref{canc}) let us start noting that, since sprt$\,b\subset B\left(
z_{0},R\right)  ,$ we can write, for any $x\in B\left(  z_{0},R\right)  $ and
$0<r_{1}<r_{2}$%
\begin{align*}
&  \int_{y\in B\left(  z_{0},R\right)  :r_{1}<d\left(  x,y\right)  <r_{2}%
}k\left(  x,y\right)  d\mu\left(  y\right)  =\\
&  =a\left(  x\right)  \int_{y\in B\left(  z_{0},R\right)  :r_{1}<d\left(
x,y\right)  <r_{2}}k_{0}\left(  y^{-1}\circ x\right)  b\left(  y\right)
d\mu\left(  y\right)  =\\
&  =a\left(  x\right)  \int_{y\in\mathbb{R}^{N+1}:r_{1}<d\left(  x,y\right)
<r_{2}}k_{0}\left(  y^{-1}\circ x\right)  b\left(  y\right)  d\mu\left(
y\right)  .
\end{align*}
Note that there exists some absolute constant $c>0\ $such that $b\left(
y\right)  $ vanishes if $x\in B\left(  z_{0},R\right)  $ and $d\left(
x,y\right)  \geq cR$; hence we can assume $r_{2}\leq cR$. Under this
condition, the last integral equals%
\begin{align*}
&  a\left(  x\right)  \int_{y\in\mathbb{R}^{N+1}:r_{1}<d\left(  x,y\right)
<r_{2}}k_{0}\left(  y^{-1}\circ x\right)  \left[  b\left(  y\right)  -b\left(
x\right)  \right]  d\mu\left(  y\right)  +\\
&  +a\left(  x\right)  b\left(  x\right)  \int_{y\in\mathbb{R}^{N+1}%
:r_{1}<d\left(  x,y\right)  <r_{2}}k_{0}\left(  y^{-1}\circ x\right)
d\mu\left(  y\right)  \equiv I+II.
\end{align*}
Now, by (\ref{hom})%
\[
\left\vert I\right\vert \leq c\left\Vert a\right\Vert _{\infty}\left\vert
b\right\vert _{\alpha}\int_{d\left(  x,y\right)  <r_{2}}\frac{d\left(
x,y\right)  ^{\alpha}}{d\left(  x,y\right)  ^{Q+2}}d\mu\left(  y\right)
=c\left\Vert a\right\Vert _{\infty}\left\vert b\right\vert _{\alpha}%
r_{2}^{\alpha}\leq c\left\Vert a\right\Vert _{\infty}\left\vert b\right\vert
_{\alpha}R_{0}^{\alpha}%
\]
while, by Lemma \ref{Lemma cancellation}$,$%
\[
\left\vert II\right\vert \leq\left\Vert a\right\Vert _{\infty}\left\Vert
b\right\Vert _{\infty}\left\vert \int_{y\in\mathbb{R}^{N+1}:r_{1}<d\left(
x,y\right)  <r_{2}}k_{0}\left(  y^{-1}\circ x\right)  d\mu\left(  y\right)
\right\vert \leq c\left\Vert a\right\Vert _{\infty}\left\Vert b\right\Vert
_{\infty}.
\]

(b) To show the existence of $h\left(  x\right)  $ let us consider, for
$0<\varepsilon_{1}<\varepsilon_{2}$ and a fixed $x\in B\left(  z_{0},R\right)
$,%
\begin{align*}
&  \int_{y\in B\left(  z_{0},R\right)  :d\left(  x,y\right)  >\varepsilon_{1}%
}k\left(  x,y\right)  d\mu\left(  y\right)  -\int_{y\in B\left(
z_{0},R\right)  :d\left(  x,y\right)  >\varepsilon_{2}}k\left(  x,y\right)
d\mu\left(  y\right) \\
&  =a\left(  x\right)  \int_{y\in B\left(  z_{0},R\right)  :\varepsilon
_{1}<d\left(  x,y\right)  <\varepsilon_{2}}k_{0}\left(  y^{-1}\circ x\right)
b\left(  y\right)  d\mu\left(  y\right) \\
&  =a\left(  x\right)  \int_{y\in\mathbb{R}^{N+1}:\varepsilon_{1}<d\left(
x,y\right)  <\varepsilon_{2}}k_{0}\left(  y^{-1}\circ x\right)  b\left(
y\right)  d\mu\left(  y\right) \\
&  =a\left(  x\right)  \int_{y\in\mathbb{R}^{N+1}:\varepsilon_{1}<d\left(
x,y\right)  <\varepsilon_{2}}k_{0}\left(  y^{-1}\circ x\right)  \left[
b\left(  y\right)  -b\left(  x\right)  \right]  d\mu\left(  y\right)  +\\
&  +a\left(  x\right)  b\left(  x\right)  \int_{y\in\mathbb{R}^{N+1}%
:\varepsilon_{1}<d\left(  x,y\right)  <\varepsilon_{2}}k_{0}\left(
y^{-1}\circ x\right)  d\mu\left(  y\right) \\
&  \equiv I+II.
\end{align*}
Now,
\begin{align*}
\left\vert I\right\vert  &  \leq\left\Vert a\right\Vert _{\infty}%
\int_{d\left(  x,y\right)  <\varepsilon_{2}}\left\vert k_{0}\left(
y^{-1}\circ x\right)  \left[  b\left(  y\right)  -b\left(  x\right)  \right]
\right\vert d\mu\left(  y\right) \\
&  \leq c\left\Vert a\right\Vert _{\infty}\left\vert b\right\vert _{\alpha
}\int_{d\left(  x,y\right)  <\varepsilon_{2}}\frac{d\left(  x,y\right)
^{\alpha}}{d\left(  x,y\right)  ^{Q+2}}d\mu\left(  y\right) \\
&  \leq c\left\Vert a\right\Vert _{\infty}\left\vert b\right\vert _{\alpha
}\varepsilon_{2}^{\alpha}%
\end{align*}
by (\ref{hom}). On the other hand,%
\[
\left\vert II\right\vert \leq\left\Vert a\right\Vert _{\infty}\left\Vert
b\right\Vert _{\infty}\left\vert \int_{y\in\mathbb{R}^{N+1}:\varepsilon
_{1}<d\left(  x,y\right)  <\varepsilon_{2}}k_{0}\left(  y^{-1}\circ x\right)
d\mu\left(  y\right)  \right\vert
\]
which tends to zero as $\varepsilon_{2}\rightarrow0,$ by Lemma
\ref{Lemma cancellation}$.$ This proves the existence of the limit $h\left(
x\right)  .$
\end{proof}

>From Theorem \ref{Thm Lp}, Proposition \ref{Prop check assumptions} and the
previous discussion, we immediately have the following:

\begin{corollary}
\label{Corollary estimate on a ball}For any fixed $z_{0}\in S,$ let
\[
Tf\left(  z\right)  =PV\int_{B\left(  z_{0},R\right)  }k\left(  z,\zeta
\right)  f\left(  \zeta\right)  d\zeta,
\]
with $k,R$ as in the previous Proposition. Then for any $p\in\left(
1,\infty\right)  $ there exists $c>0$ such that%
\[
\left\Vert Tf\right\Vert _{L^{p}\left(  B\left(  z_{0},R\right)  \right)
}\leq c\left\Vert f\right\Vert _{L^{p}\left(  B\left(  z_{0},R\right)
\right)  }%
\]
for any $f\in L^{p}\left(  B\left(  z_{0},R\right)  \right)  .$ The constant
$c$ depends on the cutoff functions $a,b$ only through their $C^{\alpha}$
norms, and does not depend on $z_{0}$ and $R.$
\end{corollary}

We still need the following covering argument:

\begin{lemma}
\label{Lemma Covering}For every $r_{0}>0$ and $K>1$ there exist $\rho
\in\left(  0,r_{0}\right)  $, a positive integer $M$ and a sequence of points
$\left\{  z_{i}\right\}  _{i=1}^{\infty}\subset S$ such that:%
\begin{align*}
S  &  \subset\bigcup\limits_{i=1}^{\infty}B\left(  z_{i},\rho\right)  ;\\
\sum_{i=1}^{\infty}\chi_{B\left(  z_{i},K\rho\right)  }\left(  z\right)   &
\leq M\text{ \ }\forall z\in S.
\end{align*}

\end{lemma}

Note that the above statement is nontrivial since the space $S$ is unbounded
and there is not a simple relation between $d$ and the Euclidean distance.
Since this property is better proved in an abstract context, we postpone its
proof to the next section, and proceed to conclude the proof of our main result:

\begin{theorem}
\label{Thm conclusion}For a suitable choice of the number $\rho_{0}$ appearing
in the definition of the kernel $k_{0}$ (see \S \ref{section nonsingular}),
for any $p\in\left(  1,\infty\right)  ,$ there exists a positive constant $c,$
depending on $p,N,p_{0},\nu$ and the matrix $B\ $such that%
\[
\left\Vert PV\left(  k_{0}\ast f\right)  \right\Vert _{L^{p}\left(  S\right)
}\leq c\left\Vert f\right\Vert _{L^{p}\left(  S\right)  }%
\]
for any $f\in L^{p}\left(  S\right)  .$
\end{theorem}

\begin{proof}
Pick a cutoff function%
\begin{align*}
A  &  \in C_{0}^{\alpha}\left(  S\right)  \text{ such that:}\\
A\left(  z\right)   &  =1\text{ for }\left\Vert z\right\Vert <\rho_{0};\\
A\left(  z\right)   &  =0\text{ for }\left\Vert z\right\Vert >2\rho_{0}%
\end{align*}
where the number $\rho_{0}$, to be fixed later, is the same appearing in the
definition of the cutoff function $\eta$ and the kernel $k_{0}$ (see
(\ref{k_0}) in \S \ref{section nonsingular}). Let%
\[
a_{i}\left(  z\right)  =A\left(  z^{-1}\circ z_{i}\right)  \text{ for
}i=1,2,...;
\]
Since $k_{0}(\zeta^{-1}\circ z)$ vanishes for $d\left(  z,\zeta\right)
>\rho_{0},$ we have that%
\begin{equation}
z\in B(z_{i},\rho_{0})\;\text{and}\;k_{0}(\zeta^{-1}\circ z)\not =%
0\;\;\Longrightarrow\;\;\zeta\in B(z_{i},C\rho_{0}) \label{def C}%
\end{equation}
for some absolute constant $C$. Define a second cutoff function%
\begin{align*}
B  &  \in C_{0}^{\alpha}\left(  S\right)  \text{ such that:}\\
B\left(  z\right)   &  =1\text{ for }\left\Vert z\right\Vert <C\rho_{0};\\
B\left(  z\right)   &  =0\text{ for }\left\Vert z\right\Vert >2C\rho_{0}%
\end{align*}
where $C$ is the constant appearing in (\ref{def C}). Let%
\[
b_{i}\left(  z\right)  =B\left(  z^{-1}\circ z_{i}\right)  \text{ for
}i=1,2,.....
\]
Note that:%
\begin{align}
\left\Vert a_{i}\right\Vert _{C^{\alpha}}  &  =\left\Vert A\right\Vert
_{C^{\alpha}}\text{ for }i=1,2,...\label{cutoff norm}\\
\left\Vert b_{i}\right\Vert _{C^{\alpha}}  &  =\left\Vert B\right\Vert
_{C^{\alpha}}\text{ for }i=1,2,...\nonumber
\end{align}
Set
\[
k_{i}(z,\zeta)=k_{0}(\zeta^{-1}\circ z)a_{i}(z)b_{i}(\zeta).
\]
Let now $R_{0}$ be as in Proposition \ref{Prop check assumptions}; set
$r_{0}=R_{0}/2C$ and let us apply Lemma \ref{Lemma Covering} for this $r_{0}$:
there exists $\rho_{0}<r_{0}$ such that
\begin{align}
S  &  \subset\bigcup\limits_{i=1}^{\infty}B\left(  z_{i},\rho_{0}\right)
;\label{covering 1}\\
\sum_{i=1}^{\infty}\chi_{B\left(  z_{i},2C\rho_{0}\right)  }\left(  z\right)
&  \leq M\text{ \ }\forall z\in S. \label{covering 2}%
\end{align}
We eventually chose this value for the constant $\rho_{0}$.

Recall that $Tf=PV(k_{0}\ast f)$. By (\ref{covering 1}) we can write%
\begin{equation}
\left\Vert Tf\right\Vert _{L^{p}\left(  S\right)  }\leq\sum_{i=1}^{\infty
}\left\Vert Tf\right\Vert _{L^{p}\left(  B\left(  z_{i},\rho_{0}\right)
\right)  }.\label{conclusion 1}%
\end{equation}
On the other side, by (\ref{def C}) for any $z\in B(z_{i},\rho_{0})$ we have%
\begin{align*}
Tf(z) &  =PV\int_{\mathbb{R}^{N+1}}k_{0}\left(  \zeta^{-1}\circ z\right)
f(\zeta)d\zeta=\\
&  =a_{i}(z)P.V.\int_{\mathbb{R}^{N+1}}k_{0}\left(  \zeta^{-1}\circ z\right)
b_{i}(\zeta)f(\zeta)d\zeta=\int_{B(z_{i},2C\rho_{0})}k_{i}(z,\zeta
)f(\zeta)d\zeta\equiv T_{i}f(z)
\end{align*}
hence%
\begin{equation}
\sum_{i=1}^{\infty}\left\Vert Tf\right\Vert _{L^{p}\left(  B\left(  z_{i}%
,\rho_{0}\right)  \right)  }=\sum_{i=1}^{\infty}\left\Vert T_{i}f\right\Vert
_{L^{p}\left(  B\left(  z_{i},\rho_{0}\right)  \right)  }.\label{conclusion 2}%
\end{equation}
Since $2C\rho_{0}\leq R_{0},$ the kernel $k_{i}$ also satisfies the
assumptions of Proposition \ref{Prop check assumptions}. Hence by Corollary
\ref{Corollary estimate on a ball} we have
\begin{equation}
\left\Vert T_{i}f\right\Vert _{L^{p}\left(  B\left(  z_{i},2C\rho_{0}\right)
\right)  }\leq c\left\Vert f\right\Vert _{L^{p}\left(  B\left(  z_{i}%
,2C\rho_{0}\right)  \right)  }\label{conclusion 3}%
\end{equation}
with $c$ independent of $i$, by (\ref{cutoff norm}). By (\ref{covering 2}) to
(\ref{conclusion 3}) and  we conclude%
\[
\left\Vert Tf\right\Vert _{L^{p}\left(  S\right)  }\leq c\sum_{i=1}^{\infty
}\left\Vert f\right\Vert _{L^{p}\left(  B\left(  z_{i},2C\rho_{0}\right)
\right)  }\leq cM\left\Vert f\right\Vert _{L^{p}\left(  S\right)  }%
\]
which ends the proof.
\end{proof}

\bigskip

\begin{proof}
[Conclusion of the proof of Theorems \ref{Thm main} and
\ref{Thm main thm strip}]Theorem \ref{Thm conclusion} and Corollary
\ref{Coroll nonsingular} imply Theorem \ref{Thm main thm strip}, by
(\ref{repr formula}). As we have shown in \S \ref{section introduction},
Theorem \ref{Thm main thm strip} in turn implies (\ref{stima L_0 a}%
)-(\ref{stima L_0 b}) in Theorem \ref{Thm main}. To finish the proof of
Theorem \ref{Thm main} we are left to prove the weak $\left(  1,1\right)
$-estimates (\ref{weak a})-(\ref{weak b}). This will be done here.

Let $u\in C_{0}^{\infty}\left(  S\right)  $. By (\ref{repr formula}) in
\S \ref{section nonsingular} we can write, for any $\alpha>0$:%
\begin{align*}
&  \left\vert \left\{  z\in S:\left\vert \partial_{x_{i}x_{j}}^{2}u\left(
z\right)  \right\vert \geq\alpha\right\}  \right\vert \leq\\
&  \leq\left\vert \left\{  z\in S:\left\vert PV\left(  k_{0}\ast Lu\right)
\left(  z\right)  \right\vert \geq\frac{\alpha}{3}\right\}  \right\vert +\\
&  +\left\vert \left\{  z\in S:\left\vert \left(  k_{\infty}\ast Lu\right)
\left(  z\right)  \right\vert \geq\frac{\alpha}{3}\right\}  \right\vert +\\
&  +\left\vert \left\{  z\in S:\left\vert c_{ij}Lu\left(  z\right)
\right\vert \geq\frac{\alpha}{3}\right\}  \right\vert \\
&  \equiv A+B+C.
\end{align*}
Now, by Corollary \ref{Coroll nonsingular}%
\[
B+C\leq\frac{3}{\alpha}\left\{  \left\Vert k_{\infty}\ast Lu\right\Vert
_{L^{1}\left(  S\right)  }+\left\Vert c_{ij}Lu\right\Vert _{L^{1}\left(
S\right)  }\right\}  \leq\frac{c}{\alpha}\left\Vert Lu\right\Vert
_{L^{1}\left(  S\right)  }.
\]
To bound $A,$ we revise as follows the proof of Theorem \ref{Thm conclusion},
writing (with the same meaning of symbols and letting $f\equiv Lu$):%
\begin{align*}
A &  =\left\vert \left\{  z\in S:\left\vert Tf\left(  z\right)  \right\vert
\geq\frac{\alpha}{3}\right\}  \right\vert \leq\\
&  \leq\sum_{i=1}^{\infty}\left\vert \left\{  z\in B\left(  z_{i},\rho
_{0}\right)  :\left\vert Tf\left(  z\right)  \right\vert \geq\frac{\alpha}%
{3}\right\}  \right\vert =\\
&  =\sum_{i=1}^{\infty}\left\vert \left\{  z\in B\left(  z_{i},\rho
_{0}\right)  :\left\vert T_{i}f\left(  z\right)  \right\vert \geq\frac{\alpha
}{3}\right\}  \right\vert \leq\\
&  \leq\sum_{i=1}^{\infty}\left\vert \left\{  z\in B\left(  z_{i},2C\rho
_{0}\right)  :\left\vert T_{i}f\left(  z\right)  \right\vert \geq\frac{\alpha
}{3}\right\}  \right\vert \leq\\
&  \leq\sum_{i=1}^{\infty}\frac{c}{\alpha}\left\Vert f\right\Vert
_{L^{1}\left(  B\left(  z_{i},2C\rho_{0}\right)  \right)  }\leq\frac
{cM}{\alpha}\left\Vert f\right\Vert _{L^{1}\left(  S\right)  }%
\end{align*}
where we used the fact that $T_{i}$ is also weak $\left(  1,1\right)  $
continuous on $L^{1}\left(  B\left(  z_{i},2C\rho_{0}\right)  \right)  ,$ by
Theorem \ref{Thm Lp}. This proves the weak estimate on the strip:%
\begin{equation}
\left\vert \left\{  z\in S:\left\vert \partial_{x_{i}x_{j}}^{2}u\left(
z\right)  \right\vert \geq\alpha\right\}  \right\vert \leq\frac{c}{\alpha
}\left\Vert Lu\right\Vert _{L^{1}\left(  S\right)  }.\label{weak on the strip}%
\end{equation}

Next, we take a cutoff function $\psi\in C_{0}^{\infty}\left(  -1,1\right)  $
such that $\psi\left(  t\right)  \geq1$ in $\left[  -\frac{1}{2},\frac{1}%
{2}\right]  $ and, for any $u\in C_{0}^{\infty}\left(  \mathbb{R}^{N}\right)
,$ apply (\ref{weak on the strip}) to $\psi u,$ getting%
\begin{align*}
&  \left\vert \left\{  x\in\mathbb{R}^{N}:\left\vert \partial_{x_{i}x_{j}}%
^{2}u\left(  x\right)  \right\vert \geq\alpha\right\}  \right\vert \leq\\
&  \leq\left\vert \left\{  \left(  x,t\right)  \in\mathbb{R}^{N}\times\left[
-\frac{1}{2},\frac{1}{2}\right]  :\left\vert \psi\left(  t\right)
\partial_{x_{i}x_{j}}^{2}u\left(  x\right)  \right\vert \geq\alpha\right\}
\right\vert \leq\\
&  \leq\left\vert \left\{  z\in S:\left\vert \partial_{x_{i}x_{j}}^{2}\left(
u\psi\right)  \left(  z\right)  \right\vert \geq\alpha\right\}  \right\vert
\leq\\
&  \leq\frac{c}{\alpha}\left\Vert L\left(  u\psi\right)  \right\Vert
_{L^{1}\left(  S\right)  }\leq\\
&  \leq\frac{c}{\alpha}\left\{  \left\Vert {\mathcal{A}}u\right\Vert
_{L^{1}\left(  \mathbb{R}^{N}\right)  }+\left\Vert u\right\Vert _{L^{1}\left(
\mathbb{R}^{N}\right)  }\right\}  .
\end{align*}
So we have proved (\ref{weak a}); then (\ref{weak b}) follows from
(\ref{weak a}) using the equation, and this ends the proof.
\end{proof}

\section{A covering lemma}

To make our proof of Theorem \ref{Thm conclusion} complete, we are left to
prove Lemma \ref{Lemma Covering}. This is what we are doing to do now, by a
general abstract argument.

\begin{definition}
We say that $\left(  X,d,\mu\right)  $ is a \emph{space of locally homogeneous
type} if the following conditions hold:

\begin{enumerate}
\item[(i)] $d:X\times X\rightarrow\mathbb{R}_{+}$ is a function such that for
some constant $C>0$

\begin{enumerate}
\item[(i$_{1}$)] For every $x,y\in X,$ if $d\left(  y,x\right)  \leq1$ then
$d\left(  x,y\right)  \leq Cd\left(  y,x\right)  $

\item[(i$_{2}$)] For every $x,y,z\in X,$ if $d\left(  x,z\right)  \leq1$ and
$d\left(  y,z\right)  \leq1$ then%
\[
d\left(  x,y\right)  \leq C\left(  d\left(  x,z\right)  +d\left(  z,y\right)
\right)  .
\]

\end{enumerate}

\item[(ii)] $\mu$ is a positive measure defined on a $\sigma$-algebra of
subsets of $X$ which contains the $d$-balls%
\[
B\left(  x,\rho\right)  =\left\{  y\in X:d\left(  y,x\right)  <\rho\right\}
,\text{ \ \ }x\in X,\rho>0.
\]

\item[(iii)] There exists $R>0$ such that if $0<R_{1}<R_{2}\leq R$ then there
exists $C=C\left(  R_{1},R_{2}\right)  $ such that%
\begin{equation}
0<\mu\left(  B\left(  x,R_{2}\right)  \right)  \leq C\mu\left(  B\left(
x,R_{1}\right)  \right)  <\infty\text{ \ for any }x\in X. \label{doubling}%
\end{equation}

\end{enumerate}
\end{definition}

\begin{remark}
Note that $\left(  S,d,dxdt\right)  $ is a space of locally homogeneous type.
Namely, condition (i) follows from Proposition \ref{Prop quasidist}, (ii)
follows from Lemma \ref{Lemma topologies} and (iii) follows from Proposition
\ref{Prop measure balls}. Hence the following theorem will imply Lemma
\ref{Lemma Covering}, and therefore will conclude the proof of Theorem
\ref{Thm main thm strip}.
\end{remark}

\begin{theorem}
Let $\left(  X,d,\mu\right)  $ be a space of locally homogeneous type. Then
for every $r_{0}>0$ and $K>1,$ there exist $\rho\in\left(  0,r_{0}\right)  $,
a positive integer $M$ and a countable set $\left\{  x_{i}\right\}  _{i\in
A}\subset X$ such that:

\begin{enumerate}
\item
\[%
{\displaystyle\bigcup\limits_{i\in A}}
B\left(  x_{i},\rho\right)  =X;
\]

\item
\[
\sum_{i\in A}\chi_{B\left(  x_{i},K\rho\right)  }\leq M^{2}.
\]

\end{enumerate}
\end{theorem}

\begin{proof}
First of all, we claim that for any $\rho>0,$ $X$ admits a maximal countable
family of disjoint balls of radius $\rho.$ Namely: the existence of a maximal
family (of arbitrary cardinality) of disjoint balls of radius $\rho$ follows
by Zorn's Lemma; let us show that this family $\left\{  B\left(  x_{\alpha
},\rho\right)  \right\}  _{\alpha}$ must be countable. Otherwise, since for
any fixed $x_{0}\in X,$%
\[
X=%
{\displaystyle\bigcup\limits_{n=1}^{\infty}}
B\left(  x_{0},n\right)  ,
\]
at least one ball $B\left(  x_{0},n\right)  $ should contain an uncountable
family of disjoint balls of radius $\rho$. By (\ref{doubling}), every such
ball has positive measure, and this would imply that $B\left(  x_{0},n\right)
$ has infinite measure, which contradicts (\ref{doubling}). This proves the claim.

Then, let $\left\{  B\left(  x_{i},\frac{\rho}{C\left(  C+1\right)  }\right)
\right\}  _{i\in A}$ be a countable maximal family of disjoint balls.

Fix $x\in X.$ There exists $i\in A$ such that%
\[
B\left(  x,\frac{\rho}{C\left(  C+1\right)  }\right)  \cap B\left(
x_{i},\frac{\rho}{C\left(  C+1\right)  }\right)  \neq\emptyset.
\]
To estimate $d\left(  x_{i},x\right)  ,$ we consider $y\in B\left(
x,\frac{\rho}{C\left(  C+1\right)  }\right)  \cap B\left(  x_{i},\frac{\rho
}{C\left(  C+1\right)  }\right)  ,$ and we find%
\[
d\left(  x_{i},x\right)  \leq C\left(  d\left(  x_{i},y\right)  +d\left(
y,x\right)  \right)  <C\left(  \frac{\rho}{C\left(  C+1\right)  }+C\cdot
\frac{\rho}{C\left(  C+1\right)  }\right)  =\rho
\]
where, to apply (i$_{1}$)-(i$_{2}$) in the definition of space of locally
homogeneous type, we have assumed $\rho\leq1$. This proves (1).

To prove (2), fix an arbitrary $i\in A;$ we want to estimate how many $j\in A$
satisfy the property%
\begin{equation}
B\left(  x_{i},K\rho\right)  \cap B\left(  x_{j},K\rho\right)  \neq\emptyset.
\label{disjoint}%
\end{equation}
Fix $x_{i}$ and $x_{j}$ and suppose there exists $y\in B\left(  x_{i}%
,K\rho\right)  \cap B\left(  x_{j},K\rho\right)  .$ We assume $K\rho\leq1;$
hence%
\[
d\left(  x_{i},x_{j}\right)  \leq C\left(  d\left(  x_{i},y\right)  +d\left(
y,x_{j}\right)  \right)  \leq C\left(  K\rho+CK\rho\right)  =C\left(
1+C\right)  K\rho,
\]
and we assume $C\left(  1+C\right)  K\rho\leq1.$ Now suppose that for
$j=1,2,...,N$ we have (\ref{disjoint}); we want to estimate $N$.

Take $z\in B\left(  x_{j},K\rho\right)  .$ Since $K\rho\leq1$ and $d\left(
x_{i},x_{j}\right)  \leq1$ we have%
\begin{align*}
d\left(  x_{i},z\right)   &  \leq C\left(  d\left(  x_{i},x_{j}\right)
+d\left(  x_{j},z\right)  \right)  \leq C\left(  C\left(  1+C\right)
K\rho+K\rho\right)  =\\
&  =K\rho\left(  C\left(  C^{2}+C+1\right)  \right)  \equiv K^{\prime}\rho
\end{align*}
with $K^{\prime}>1.$ The previous computation shows that%
\[%
{\displaystyle\bigcup\limits_{j=1}^{N}}
B\left(  x_{j},K\rho\right)  \subset B\left(  x_{i},K^{\prime}\rho\right)  .
\]
Since by construction the balls $B\left(  x_{i},\frac{\rho}{C\left(
C+1\right)  }\right)  $ are pairwise disjoint, we have%
\begin{align*}
\sum_{j=1}^{N}\mu\left(  B\left(  x_{j},\frac{\rho}{C\left(  C+1\right)
}\right)  \right)   &  =\mu\left(
{\displaystyle\bigcup\limits_{j=1}^{N}}
B\left(  x_{j},\frac{\rho}{C\left(  C+1\right)  }\right)  \right)  \leq\\
&  \leq\mu\left(
{\displaystyle\bigcup\limits_{j=1}^{N}}
B\left(  x_{j},K\rho\right)  \right)  \leq\mu\left(  B\left(  x_{i},K^{\prime
}\rho\right)  \right)  .
\end{align*}
Assuming $K^{\prime}\rho\leq R$ (with $R$ as in (iii) of the above
definition), we also have, for some constant $M$ only depending on $C\ $and
$\rho,$%
\[
\sum_{j=1}^{N}\mu\left(  B\left(  x_{j},\frac{\rho}{C\left(  C+1\right)
}\right)  \right)  \leq M\mu\left(  B\left(  x_{i},\frac{\rho}{C\left(
C+1\right)  }\right)  \right)  .
\]
Now fix any $j=1,2,...,N.$ Note that $i$ satisfies (\ref{disjoint}); repeating
the previous argument exchanging $i$ with $j$ we get%
\[
\mu\left(  B\left(  x_{i},\frac{\rho}{C\left(  C+1\right)  }\right)  \right)
\leq\mu\left(  B\left(  x_{j},K^{\prime}\rho\right)  \right)  \leq M\mu\left(
B\left(  x_{j},\frac{\rho}{C\left(  C+1\right)  }\right)  \right)  .
\]
We have found that%
\[
\sum_{k=1}^{N}\mu\left(  B\left(  x_{k},\frac{\rho}{C\left(  C+1\right)
}\right)  \right)  \leq M^{2}\mu\left(  B\left(  x_{j},\frac{\rho}{C\left(
C+1\right)  }\right)  \right)  \text{ for any }j=1,2,...,N.
\]
Letting
\[
a=\min_{j=1,2,...,N}\mu\left(  B\left(  x_{j},\frac{\rho}{C\left(  C+1\right)
}\right)  \right)
\]
we get%
\[
Na\leq\sum_{k=1}^{N}\mu\left(  B\left(  x_{k},\frac{\rho}{C\left(  C+1\right)
}\right)  \right)  \leq M^{2}a
\]
and since, by (iii), $0<a<\infty,$ we infer $N\leq M^{2},$ which is (2),
provided $\rho$ satisfies all the conditions we have imposed so far:%
\[
\rho\leq1;\text{ }K\rho\leq1;K^{\prime}\rho\equiv K\rho\left(  C\left(
1+C\left(  1+C\right)  \right)  \right)  \leq R.
\]

\end{proof}


\bigskip

\begin{thebibliography}{99}                                                                                               %


\bibitem {B}M. Bramanti: Singular integrals on nonhomogeneous spaces: $L^{2}$
and $L^{p}$ continuity from H\"{o}lder estimates. \textit{To appear on
}Revista Matematica Iberoamericana.

\bibitem {BB}M. Bramanti, L. Brandolini: Schauder estimates for parabolic
nondivergence operators of H\"{o}rmander type. Journal of Differential
Equations, 234 (2007), no.1, 177-245.

\bibitem {CG}A. Chojnowska-Michalik, B. Goldys: Symmetric Ornstein-Uhlenbeck
semigroups and their generators. Probab. Theory Related Fields 124 (2002), no.
4, 459--486.

\bibitem {CPP}C. Cinti, A. Pascucci, S. Polidoro: Pointwise estimates for a
class of non-homogeneous Kolmogorov equations. Math. Ann. 340 (2008), no. 2, 237--264.

\bibitem {CW}R. Coifman, G. Weiss: Analyse Harmonique Non-Commutative sur
Certains Espaces Homogenes. Lecture Notes in Mathematics, n.~242.
Springer-Verlag, Berlin-Heidelberg-New York, 1971.

\bibitem {DL}G. Da Prato, A. Lunardi: On the Ornstein-Uhlenbeck operator in
spaces of continuous functions. J. Funct. Anal. 131 (1995), no. 1, 94--114.

\bibitem {DZ}G. Da Prato, J. Zabczyk: Stochastic equations in infinite
dimensions. Encyclopedia of Mathematics and its Applications, 44. Cambridge
University Press, Cambridge, 1992.

\bibitem {DZ1}G. Da Prato, J. Zabczyk: Second order partial differential
equations in Hilbert spaces. London Mathematical Society Lecture Note Series,
293. Cambridge University Press, 2002.

\bibitem {DP}M. Di Francesco, S. Polidoro: Schauder estimates, Harnack
inequality and Gaussian lower bound for Kolmogorov-type operators in
non-divergence form. Adv. Differential Equations 11 (2006), no. 11, 1261--1320.

\bibitem {FL}B. Farkas, A. Lunardi: Maximal regularity for Kolmogorov
operators in $L^{2}$ spaces with respect to invariant measures. J. Math. Pures
Appl. (9) 86 (2006), no. 4, 310--321.

\bibitem {Fo}G.~B.~Folland: Subelliptic estimates and function spaces on
nilpotent Lie groups, Arkiv for Mat. 13, (1975), 161-207.

\bibitem {H}L. H\"{o}rmander: Hypoelliptic second order differential
equations. Acta Math. 119 (1967) 147--171.

\bibitem {Kol}A. Kolmogoroff: Zuf\"{a}llige Bewegungen (zur Theorie der
Brownschen Bewegung). Ann. of Math. (2) 35 (1934), no. 1, 116--117.

\bibitem {K}L. P. Kupcov: The fundamental solutions of a certain class of
elliptic-parabolic second order equations. (Russian) Differencial'nye
Uravnenija 8 (1972), 1649--1660, 1716.

\bibitem {LP}E. Lanconelli, S. Polidoro: On a class of hypoelliptic evolution
operators. Partial differential equations, II (Turin, 1993). Rend. Sem. Mat.
Univ. Politec. Torino 52 (1994), no. 1, 29--63.

\bibitem {LPP}E. Lanconelli, A. Pascucci, S. Polidoro: Linear and nonlinear
ultraparabolic equations of Kolmogorov type arising in diffusion theory and in
finance. Nonlinear problems in mathematical physics and related topics, II,
243--265, Int. Math. Ser. (N. Y.), 2, Kluwer/Plenum, New York, 2002.

\bibitem {Lu}A. Lunardi: Schauder estimates for a class of degenerate elliptic
and parabolic operators with unbounded coefficients in $\mathbb{R}^{N}.$ Ann.
Sc. Norm. Sup. Pisa Cl. Sci. (4), 24 (1997), pp. 133-164.

\bibitem {Lu2}A. Lunardi: On the Ornstein-Uhlenbeck operator in $L^{2}$ spaces
with respect to invariant measures. Trans. Amer. Math. Soc. 349 (1997), no. 1, 155--169.

\bibitem {MPRS}G. Metafune, J. Pr\"{u}ss, A. Rhandi, R. Schnaubelt: The domain
of the Ornstein-Uhlenbeck operator on an $L^{p}$-space with invariant measure.
Ann. Sc. Norm. Super. Pisa Cl. Sci. (5) 1 (2002), no. 2, 471--485.

\bibitem {NTV}F. Nazarov, S. Treil, A. Volberg: Weak type estimates and Cotlar
inequalities for Calder\'{o}n-Zygmund operators on nonhomogeneous spaces.
Internat. Math. Res. Notices 1998, no. 9, 463--487.

\bibitem {P}A. Pascucci: Kolmogorov equations in physics and in finance.
Elliptic and parabolic problems, 353--364, Progr. Nonlinear Differential
Equations Appl., 63, Birkh\"{a}user, Basel, 2005.

\bibitem {PZ}E. Priola, J. Zabczyk: Liouville theorems for non-local
operators, J. Funct. Anal. 216 (2004), no. 2, 455-490.

\bibitem {RS}L.~P.~Rothschild, E.~M.~Stein: Hypoelliptic differential
operators and nilpotent groups. Acta Math., 137 (1976), 247-320.

\bibitem {tERS}A. F. M. ter Elst, D. W. Robinson, A. Sikora: Riesz transforms
and Lie groups of polynomial growth. J. Funct. Anal. 162 (1999), no. 1, 14--51.

\bibitem {OU}G. E. Uhlenbeck, L. S. Ornstein: On the Theory of the Browninan
Motion. Phys. Rev. vol. 36, n.3 (1930), 823-841. This paper is also contained
in: N. Wax (ed.): Selected Papers on Noise and Stochastic Processes. Dover, 2003.
\end{thebibliography}
\end{document}